\documentclass{article}[12pt]
\usepackage[T2A]{fontenc}
\usepackage[utf8]{inputenc}

\usepackage{hyphenat}

\usepackage{amsmath}
\usepackage{amssymb}
\usepackage{amsfonts}
\usepackage{amscd,euscript,amsxtra,amsthm}

\newtheorem{theorem}{Theorem}

\newtheorem{proposition}{Proposition}
\newtheorem{lemma}{Lemma}[section]

\usepackage{amscd,euscript,amsxtra,amsthm}

\theoremstyle{definition}
\newtheorem{definition}{Definition}
\newtheorem{convention}{Convention}

\theoremstyle{remark}
\newtheorem{remark}{Remark}
\theoremstyle{remark}
\newtheorem{example}{Example}

\theoremstyle{remark}

\usepackage[all]{xy}

\numberwithin{equation}{section}

\def\A{{{\mathbb A}}}
\def\I{{{\mathbb I}}}

\def\E{{{\mathbb E }}}
\def\D{{{\mathbb D }}}
\def\C{{{\mathbb C }}}

\def\P{{{\mathbb P }}}
\def\F{{{\mathbb F }}}

\def\L{{{\mathbb L }}}

\def\FF{{{\mathcal F}}}

\def\JJ{{{\mathcal J}}}

\def\MM{{{\mathcal M}}}
\def\NN{{{\mathcal N}}}
\def\OO{{{\mathcal O}}}

\def\EExt{{{\mathcal E}xt}}

\def\FFitt{{{\mathcal F}itt}}

\def\Spec{{{\rm Spec \,}}}
\def\Sing{{{\rm Sing \,}}}

\def\Proj{{{\rm Proj \,}}}
\def\Hilb{{{\rm Hilb \,}}}
\def\Univ{{{\rm Univ \,}}}

\def\Supp{{{\rm Supp \,}}}

\def\dim{{{\rm dim \,}}}

\def\hd{{{\rm hd \,}}}
\def\coker{{{\rm coker \,}}}
\def\ker{{{\rm ker \,}}}
\def\rank{{{\rm rank \,}}}
\def\im{{{\rm im \,}}}

\def\tors{{{\rm tors \,}}}

\def\red{{{\rm red}}}
\def\Ass{{{\rm Ass\,}}}

\title
{Moduli of Admissible Pairs for Arbitrary Dimension, I: Resolution}

\author{{Nadezhda V. Timofeeva\footnote{Center of Integrable Systems, P.G. Demidov Yaroslavl State University, Russia}}}
\date{}

\begin{document}
\large
\maketitle

\renewcommand{\refname}{References}
\renewcommand{\proofname}{Proof.}
\thispagestyle{empty}

\maketitle{}\vspace{1cm} 
\begin{quote}
\noindent{\sc Abstract. } A procedure resolving a family of  torsion-free coherent sheaves on a non-singular $N$-dimensional projective algebraic variety into a family of locally free sheaves on projective schemes of certain class is proposed. This is a higher-dimensional analog of the resolution (called the standard resolution in previous works of the author) for a family of coherent sheaves on a surface. The method is applicable to all existing flat families of torsion-free sheaves including those who does not contain locally free sheaves.

Bibliography: 23 items.
\medskip

 \end{quote}

\bigskip

\section{Introduction}
In the present paper we do preparatory work for constructing a scheme of moduli of
semistable admissible pairs for arbitrary dimension $> 2$. In the series
of articles \cite{Tim0}--\cite{Tim10} the construction was done in
dimension two: admissible schemes had dimension two and were obtained
by some transform\-ation from a nonsingular projective algebraic surface
$S$. The final result was an isomorphism between the moduli functor of
Gieseker-semistable torsion-free coherent sheaves of rank $r$ and with
Hilbert polynomial $rp(n)$ on the surface $S$ with a fixed
polarization $L$ and the moduli functor of semistable admissible pairs
$((\widetilde S, \widetilde L), \widetilde E)$. Each such a pair consists
of a scheme $\widetilde S$ of some fixed class with a distinguished polarization
$\widetilde L$ and a semistable locally free sheaf $\widetilde E$ of
rank $r$ and with Hilbert polynomial $rp(n)$. The schemes $\widetilde S$ were called admissible schemes. The advantage of the moduli functor of semistable admissible pairs is that  its moduli scheme is isomorphic to the 
Gieseker--Maruyama moduli scheme for the same rank and Hilbert
polynomial. In particular, this provides a
compactification of the moduli space of stable vector bundles by vector
bundles on some special (admissible) schemes instead of the classical
compactification by attaching non-locally free coherent sheaves.
Applying the procedure of the resolution, in the sequel we will develop a functorial approach to constructing of a moduli scheme for such  vector bundles on projective schemes. The whole of the construction generalizes the
analogous results for the two-dimensional case (\cite{Tim10}).

The basement for the subject of interest is the Kobayashi--Hitchin
correspondence which is discussed below. It allows one to apply algebraic geometrical methods to
problems of differential geometrical or gauge theoretical setting by
transferring consideration of the moduli of connections in a vector
bundle (including vector bundles endowed with additional structures)
to consideration of the moduli for vector bundles which are
slope-stable.

 Consider a compact complex algebraic variety $X$,  a complex vector
 bundle $E$ on $X$, Hermitian metrics $g$ on $X$ and $h$ on $E$
 respectively. A holomorphic vector bundle on a compact complex manifold
 with a Hermitian metric is {\it stable}
\cite{LuTel} if and only if $E$ admits an irreducible $g$-Hermitian--Einstein metric. If $E$ is a differentiable vector bundle and  $\bar
\delta$ is a holomorphic structure in  $E$ then  $\bar \delta$ is called
 $g$-{\it stable} \cite{LuTel}  if the holomorphic bundle $(E, \bar \delta)$
 has this property. Let $M^{st}_g(E)$ be the  moduli space for isomorphy classes of $g$-stable
 holomorphic structures on $E$ and $M^{HE}_g(E,h)$ be the moduli space for
 gauge equivalence classes of irreducible $h$-unitary $g$-Hermitian--Einstein
 connections in $E$.
Many mathematicians  (S.~Donaldson,
S.~Kobayashi,
N.~Hitchin, F.~Bogomolov,
N.~Buchdahl, 
M.~Itoh and H.~Nakajima,
J.~Li, K.~Uhlenbeck and S.-T.~Yau
and others) developed the theory in various aspects; its central
result can be formulated as follows (the Kobayashi--Hitchin correspondence) \cite{LuTel}:\\
{\it There exists a real-analytic isomorphism} $M^{HE}_g(E,h) \cong
M^{st}_g(E)$.

 If $X$ is a projective manifold and the Hodge metric  $g$ is associated
 to the ample invertible sheaf~$L$ which induces the immersion of~$X$
 into the projective space then the definition of
$g$-stability coincides with the definition of $L$-slope-stability
of a vector bundle~$E$ \cite{LuTel}.

In the case of rank 2 vector bundles on a surface J.~Li proved
\cite{Li93} that the Donaldson--Uhlenbeck compactification for the moduli
space  (of gauge equivalence classes) of anti-self-dual connections
admits  a complex structure which provides it with a reduced
projective scheme structure. The Gieseker--Maruyama
compactification for the scheme of moduli for stable vector bundles
has a morphism to the compactified scheme of anti-self-dual connections.

Moduli schemes for stable vector bundles on an algebraic variety
usually are nonprojective (and noncompact)  \cite{Mar1,Mar}, and
for application of methods of algebraic geometry it is useful to
include the scheme (or variety) of moduli for vector bundles into
some appropriate projective scheme as its open subscheme. This
problem is traditionally called a problem of compactification of
a moduli space. There arises a natural question of searching for
compactificat\-ions of moduli for vector bundles and of moduli for
connections which admit expanding of the Kobayashi--Hitchin
correspondence to  both compactifications in whole.

The classical Gieseker--Maruyama moduli space for coherent sheaves
includes classes of co\-herent non-locally free sheaves, and hence it
is not suitable for this purpose. Also the following compactifications
(related to the Yang--Mills field theory) are known: the Donaldson--Uhlenbeck compactification \cite{Donaldson89} (involving so-called
ideal connections), the algebraic geometrical functorial approach to its
construction given by V.~Baranovsky \cite{Bar}, the algebraic geometrical
version for framed sheaves done by U.~Bruzzo, D.~Markushevich and
A.~Tikhomirov in \cite{BruMarTikh}, and the Taubes--Uhlenbeck--Feehan compactification \cite{Feehan}. Also there is the
algebraic geometrical bubble-tree compactification announced for
vector bundles on  a surface and constructed for rank 2 vector
bundles by D.~Markushevich, A.~Tikhomirov and G.~Trautmann in 2012 \cite{MTT}. A different approach is developed by the author of the 
present paper for the 2-dimensional case (see \cite{Tim10} and references
therein), and here we describe a first step in the construction of its
multidimensional analog. The next steps are to be described in the next paper.

The compactifications which do not involve non-bundles can be of use for
extend\-ing  the Kobayashi--Hitchin cor\-respond\-ence to the compact case.
Since the Kobayashi--Hitchin correspondence holds not only in
dimension 2, we extend the ``locally free'' compactification to
arbitrary dimension. The Kobayashi--Hitchin correspondence operates
with the coefficient field~$\C$. We work with an arbitr\-ary
algebraically closed field~$k$ of zero characteristic, because the
isomorphism of functors is not subject to special properties of~$\C$.

Let $S$ be a nonsingular (irreducible) projective algebraic variety over an algebraic\-ally closed field $k=\overline k$ of zero characteristic, $\OO_S$ be its structure sheaf, $E$ be a coherent torsion-free $\OO_S$-module, $E^{\vee}:={{\mathcal H}om}_{{\mathcal O}_S}(E, {\mathcal O}_S)$ be its dual ${\mathcal O}_S$-module.  A locally free sheaf and its corresponding vector bundle are canonically identified with each other, and both terms are used as synonyms. Let  $L$ be a very ample invertible sheaf on $S$; it is fixed and is used as a polarization. The symbol $\chi(\cdot)$ denotes the Euler--Poincar\'{e} characteristic of a coherent sheaf, $c_i(\cdot)$ is its $i$-th Chern class.

Recall the definition of a sheaf of 0-th Fitting ideals known from
commutat\-ive algebra. Let $X$ be a scheme and $F$ be an $\OO_X$-module with a
finite presentation $$\widehat F_1 \stackrel{\varphi}{\longrightarrow} \widehat F_0
\to F\to 0.$$ Without loss of generality we assume that  $\rank \widehat F_1
\ge \rank \widehat F_0$.
\begin{definition} {\it The sheaf of 0-th Fitting ideals} of
the $\OO_X$-module $F$ is defined as  $$\FFitt^0 F =\im
(\bigwedge^{\rank \widehat F_0} \widehat F_1 \otimes \bigwedge^{ \rank  \widehat F_0}
\widehat F_0^{\vee} \stackrel{\varphi'}{\longrightarrow}\OO_X),$$ where
$\varphi'$ is the morphism of  $\OO_X$-modules induced by~$\varphi$.
\end{definition}

\begin{remark} We  work with invertible sheaves $L$ and
$\widetilde L$, where the expression for $\widetilde L$ involves
an appropriate tensor power of $L$. In the further considerations we replace $L$ by its big enough tensor power, if necessary, in order to make $\widetilde L$ very ample.
 This power can be chosen uniform
and fixed, as shown in \cite{Tim4} for the case when the homological dimension of $E$ equals one: $\hd E=1$. The same reasoning is true for higher dimensions of $S$ and for arbitrary finite homological dimensions of sheaves. All Hilbert polynomials are
computed according to new $L$ and $\widetilde L$ respectively.
\end{remark}

As shown in \cite{Tim3}, in the case $\dim S= 2$ whenever  the admissible scheme $\widetilde S$ is not isomorphic to $S$, it has the form $$\widetilde S=\Proj \bigoplus_{s\ge
0}(I[t]+(t))^s/(t^{s+1})$$ and is decomposed into the finite union of several
components $\widetilde S=\bigcup_{i\ge 0}\widetilde S_i$. It has a
morphism $\sigma\colon\widetilde S \to S$ which is induced by the
structure of $\OO_S$-algebra on the graded ring \linebreak $\bigoplus_{s\ge
0}(I[t]+(t))^s/(t^{s+1})$. 
In the present paper we develop
a method to resolve singularities of a torsion-free coherent
sheaf in higher dimensions by an iterative process, which leads to an
analog of an admissible scheme for any dimension. The standard
resolution and the notion of admissible scheme for the 2-dimensional case are included in our
method as the simplest nontrivial case.

\smallskip

In the present paper we prove the following result.
\begin{theorem}\label{thres} Let $T$ be an algebraic scheme of finite type over an algebraically closed field  $k$ of characteristic~0 and $L$ be  a  \textup{(}big enough\textup{)} very ample invertible sheaf on a nonsingular projective algebraic variety $S$. Set $\Sigma=T\times S$.
Let $\mathbb E$ be a  $T$-flat coherent   ${\mathcal O}_\Sigma$-module of rank~$r$ such that  $\mathbb E|_{t\times S}$ is a torsion free coherent sheaf, and a fibrewise Hilbert polynomial $\chi(\mathbb E|_{t\times S} \otimes L^n)=rp(n)$ does not depend on the choice of the point  $t\in T.$  
The there exist

$\bullet$ a scheme morphism $\xi \colon \widetilde T \to T$,

$\bullet$ a flat family $\pi\colon\widetilde \Sigma \to \widetilde T$ with a scheme morphism $\sigma\!\!\!\sigma\colon\widetilde \Sigma \to \Sigma$,

$\bullet$ an invertible $\mathcal O_{\widetilde \Sigma}$-sheaf $\widetilde {\mathbb L}$,\\
such that 

$\bullet$ $\widetilde {\mathbb E}=\sigma\!\!\!\sigma^{\Delta}{\mathbb E}$ is a locally free  ${\mathcal O}_{\widetilde \Sigma}$-module,

$\bullet$ $\chi(\widetilde {\mathbb L}^n|_{\pi^{-1}(t)})$ is uniform over $t\in T$,

$\bullet$ $\chi (\widetilde {\mathbb E} \otimes \widetilde {\mathbb L}^n|_{\pi^{-1}(t)})=rp(n)$, if the family  $\E$ contains sheaves which are deformation equivalent to locally free sheaves,

$\bullet$ the following diagram commutes:
$$\xymatrix{ \Sigma \ar[d]_p&\ar[l]_{\sigma\!\!\!\sigma}\widetilde \Sigma \ar[d]^\pi\\
T&\ar[l]_{\xi}\widetilde T
}
$$
\end{theorem}
The symbol $\sigma\!\!\!\sigma^{\Delta}$ denotes the pure transform  (cf. Definition \ref{puretrans}) of a sheaf under the morphism $\sigma\!\!\!\sigma$. The pure transform is expressed as  $\sigma\!\!\!\sigma^{\Delta} \E=\sigma\!\!\!\sigma^{\ast}\E/\tors$, where the subsheaf   $\tors$ generalizes the usual torsion subsheaf and is defined in Section \ref{s5}.
The resolution described in the present paper takes any coherent torsion-free $\OO_S$-sheaf $E$ on a polarized projective scheme $S$ to a pair of the form $((\widetilde S, \widetilde L),\widetilde E)$, which is a generalization of admissible pair in the dimension two (resp., homological dimension 1 for the sheaf  $E$) case. Here $(\widetilde S,\widetilde L)$ is projective polarized scheme; its structure depends on the structure of the sheaf $E$ under the resolution. In the next paper we are intended to develop a functorial approach to a moduli scheme of such pairs. In particular case when $\widetilde S=S$, such a pair is called {\it $S$-pair}. 
\begin{convention}\label{noembed} We assume that {\it the base $T$ of the family of sheaves $\E$ under the resolution contains no embedded components.} This assumption implies that the sheaf $\OO_T$ and locally free $\OO_T$-sheaves are pure.
\end{convention}

Using the results of the present paper, in the forthcoming paper we will prove that a moduli scheme $\widetilde M$ of admissible pairs, which are deformation equivalent to $S$-pairs, is isomorphic to the union of those components of the Gieseker -- Maruyama moduli scheme for semistable coherent sheaves with the same Hilbert polynomial, which contain locally free sheaves. 

In Section \ref{s2} we give general remarks on the procedure of resolution and families to be constructed. Also we give there some important correction and explanation for the previous author's papers \cite{Tim8, Tim10}. Section~\ref{s3} plays the central role and contains the motivation and the description of the resolution. We give a transformation of a family of coherent torsion-free sheaves~$(T, \L,\E)$ with base scheme $T$ into a family of admissible semistable pairs~$((\widetilde T, \pi\colon \widetilde \Sigma \to \widetilde T,  \widetilde \L),\widetilde E)$.  Section~\ref{s4} is devoted to the structure of admissible schemes and to their polarizations. Finally, Section \ref{s5} contains an additional result concerning admissible pairs. It clarifies the structure of admissible schemes and sheaves on them. Also in Section \ref{s5} we give a description for the subsheaf~$\tors$ in the multidimensional case.

{\bf Acknowledgements.} 
This work was carried out within the framework of a development program for the Regional Scientific and Educational Mathematical Center of the P.G. Demidov Yaroslavl State University with financial support from the Ministry of Science and Higher Education of the Russian Federation (Agreement on provision of subsidies from the federal budget No.  075-02-2025-1636). 

Author expresses her sincere gratitude to the anonymous referee for pointing out of a mistake in the previous version of this paper. Author is very grateful to A.I.~Bondal (V.A.~Steklov Institute) for important comments and useful conversations. Also author thanks 
Ye.A.~Medvedev for checking algebraic (ring-theoretic) counterparts of conjectures concerning the behavior of torsions  of sheaves under the morphisms which arise in the resolution.
\smallskip

\section{The procedure of resolution for a sheaf. The procedure of resolution for a family of sheaves. Preliminaries } \label{s2}
{\it Terminological conventions.} In the whole of the text of the present article a {\it scheme} is a Noetherian algebraic scheme of finite type over an algebraically closed field of zero characteristic. The scheme can be non-reduced if the opposite is not stated. The scheme morphism  $f\colon X\to Y$ is called {\it birational} if there are non-empty open subschemes $X_0\subset X$ and $Y_0 \subset Y$ such that the morphism $f$ includes into the commutative diagram  $$\xymatrix{X\ar[r]^f&Y\\
\ar@{^(->}[u] X_0 \ar[r]^{f_0}&Y_0 \ar@{^(->}[u]}
$$
where $f_0$ is  an isomorphism.

Let $S$ be a scheme over a field   $k$ and $E$ be a coherent $\OO_S$-module. The following procedure is mentioned as {\it a resolution of the sheaf} $E$: 
\begin{itemize}
    \item a surjective morphism of $k$-schemes $\sigma\colon \widetilde S \to S.$ In this paper  all the morphisms $\sigma$  are birational;
    \item a transformation taking the $\OO_S$-module $E$ to a locally free $\OO_{\widetilde S}$-module $\widetilde E$. In the present paper the transformations are of the form $E \mapsto \sigma^\ast E/\tors$, where $\tors$ is either the usual torsion sheaf or some generalization of the torsion to be defined below (Section \ref{s5}). 
\end{itemize}

Let $\E$ be a family of coherent sheaves on fibres of the family of projective  $k$-schemes $p\colon \Sigma \to T$. 
{\it A resolution of the family of sheaves} $\E$ is the following procedure: 
\begin{itemize}
    \item a surjective morphism of schemes  $\sigma\!\!\!\sigma\colon \widetilde \Sigma \to \Sigma.$ In the present paper the morphisms $\sigma\!\!\!\sigma$ are birational as well as their analogs for single sheaves;
		\item a birational morphism  $\xi\colon \widetilde T \to T$;
    \item a transformation which takes  $\OO_\Sigma$-module $\E$ to a locally free  $\OO_{\widetilde \Sigma}$-module  $\widetilde \E$. In the present paper the transformations are of the form $\E \mapsto\widetilde \E:=\sigma\!\!\!\sigma^\Delta \E$. 
\end{itemize}

We construct only those resolutions which take flat families of schemes to flat families of schemes. This means that starting from a flat family of schemes   $p\colon \Sigma \to T$ (we start from a trivial family $\Sigma= T\times S$ with the usual projection $p\colon T\times S \to T$, forgetting the second factor) carrying a $T$-flat sheaf of $\OO_{\Sigma}$-modules $\E$, we arrive to a flat family of schemes  $\pi\colon\widetilde \Sigma \to \widetilde T$ which includes into the commutative square 
\begin{equation} \label{tri}
  \xymatrix{\widetilde \Sigma \ar[r] \ar[d]_{\widetilde \pi}&\Sigma \ar[d]^p\\
\widetilde T \ar[r]^\xi&T}  
\end{equation}
The transformation of the base $T$ to $\widetilde T$ can be non-identity in order to the resulting family $\widetilde \pi \colon \widetilde \Sigma \to \widetilde T$ be flat.

We will produce and study the resolution of certain type which we call {\it a standard resolution}. Let $\widetilde t\in \widetilde T$ and $t=\xi (\widetilde t) \in T$ be closed points; then the resolution procedure takes the coherent $\OO_S$-module  $E_t:=\E|_{p^{-1}(t)}$ on the fibre  $p^{-1}(t)=\{t\}\times S \subset \Sigma=T\times S$ to a locally free coherent  $\OO_{\widetilde S_{\widetilde t}}$-module $\widetilde E_{\widetilde t}:=\widetilde \E|_{\widetilde \pi^{-1}(\widetilde t)}$ on the fibre  $\widetilde \pi^{-1}(\widetilde t)=\widetilde S_{\widetilde t} \subset \widetilde \Sigma$.   

As an example (and as a model particular case) of such a transformation we consider non-invertible sheaf of ideals $I\subset \OO_S$ and a scheme of the form $\Proj \bigoplus_{s\ge 0} (I[t]+(t))^s/(t)^{s+1}$. It can be included into a flat family of schemes whose generic fibre is isomorphic to  $S$ (initially $S$ is non-singular variety; in the case when several transformations are composed, beginning with the second step it can be necessary to work with a projective scheme which is not a variety). To build up a family take a product $\Sigma=T\times S$ for $T=\P^1$, fix any closed point $0\in T$ and denote by~$i_0$ the closed immersion  $0\times S \hookrightarrow \Sigma$. Let $\widetilde \Sigma$ be a scheme which is obtained by blowing up of the product $\Sigma$ in the sheaf of ideals $\I:=\ker (\OO_\Sigma \to i_{0\ast}(\OO_S/I))$. The scheme   $\widetilde \Sigma$ is included into the diagram~(\ref{tri}) where $\xi$ is an isomorphism. The morphism $\widetilde\pi$ is projective.  It factors through a closed immersion into an appropriate relative projective space $\widetilde \pi\colon \widetilde \Sigma \stackrel{i}{\hookrightarrow }\P_T \to T$.  For  $\Sigma_0:=\Sigma \setminus (0\times S)$ we have $\widetilde \Sigma=\overline{i(\Sigma_0)}$. Then by  \cite[Ch. III, Proposition 9.8]{Hart}, $\widetilde \Sigma$ is a scheme which is flat over~$T$. Let $\L$ be an invertible $\OO_\Sigma$-sheaf which is very ample relative to~$T$, and let $\L$ be big enough for the sheaf $\widetilde \L=\sigma\!\!\!\sigma^\ast \L \otimes \sigma\!\!\!\sigma^{-1} \I \cdot \OO_{\widetilde \Sigma}$ also to be very ample relative to~$T$. Then the sheaf $\widetilde \pi_\ast (\widetilde \L^n)$ is locally free and  $\rank \widetilde \pi_\ast (\widetilde \L^n) =\rank p_\ast (\L^n)$. 

Also if $t$ is a parameter along  $T$, then  $\widetilde \pi^{-1}0= \Proj \bigoplus_{s\ge 0} (I[t]+(t))^s/(t)^{s+1}.$ The computation of the fibre  $\widetilde \pi^{-1}0$ is transferred directly from \cite{Tim3}, where it was done for  $S$ being a surface, but the argument holds for any scheme $S$ over the field $k$ and for any non-invertible ideal sheaf $I\subset \OO_{0 \times S}$. The fact that the singular locus of the coherent sheaf $E$ under the resolution in  \cite{Tim3} is zero-dimensional, is irrelevant in deducing of the formula for the fibre of the family with one-dimensional base when the family is obtained by a single blowing up with the center in one fibre.

But if the base  $T$ is not a curve, then the preserving of the flatness of the family of schemes after the transformation needs special efforts.  For example set $T=\A^2$, $S$ a surface and the blowing up  done with the centre in a point $0\times p \in T\times S$. Then the fibre $\widetilde S \subset \widetilde{T\times S}$ over the point $0\in T$ is not equidimensional scheme: it contains 3-dimensional exceptional divisor of the blowup morphism as its component. This difficulty is eliminated by a certain transformation of the base of the family. We describe this base transformation below. In the previous papers \cite{Tim8, Tim10} devoted to the case when $S$ is a nonsingular projective surface (and respectively $\hd \E=1$), this circumstance was lost. Also in these papers there were incorrect applications of the definition of flatness in dimensions $\ge n$ \cite[part 1, definition 5.2.1]{RG} and theorem \cite[theorem 5.2.2]{RG} which is based on it. In the present paper we consider the case of a variety $S$ of arbitrary dimension ($\ge 2$). By this reason families of coherent sheaves $\E$  are assumed to be of arbitrary finite homological dimension. Then the case of sheaves on a surface is also included. The procedure of resolution of the family of sheaves on a surface, as it is described in  \cite{Tim8, Tim10}, is preserved, but it must be performed after the base transformation to be described in \S \ref{s3}. Also in \cite{Tim8, Tim10} the expression $\sigma\!\!\!\sigma^{\ast}\E/\tors$ is to be understood as the pure transform of the sheaf $\E$, i.e. the quotient of the sheaf $\sigma\!\!\!\sigma^\ast \E$ under an appropriate subsheaf in the subsheaf of torsion in the modified sense. Changes and corrections which are necessary in the functorial construction of the moduli space for admissible pairs of dimension 2 will be described in the next paper. There we will discuss the analogous construction for arbitrary dimension $\ge 2$.

Also note that if some (not obligatory one-parameter) flat family of schemes $\widetilde \pi\colon\widetilde \Sigma \to \widetilde T$ carries a locally free $\OO_{\widetilde \Sigma}$-sheaf $\widetilde \E$ and an invertible sheaf $\widetilde \L$, which is very ample over  $\widetilde T$, then direct images  $\widetilde \pi_\ast (\widetilde \E \otimes \widetilde \L^n)$ are locally free for  $n\gg 0$ and hence the sheaf $\widetilde \E$ has fibrewise Hilbert polynomial with respect to  $\widetilde \L$ which is the same for all fibres.

\section{The Sheaves-to-Pairs transformation (the standard
resolution)}\label{s3}

The aim of this section is to modify the transformation of a family of  coherent torsion-free sheaves to a
family $((\widetilde S, \widetilde
L), \widetilde E)$ described in Theorem \ref{thres}. The initial family of sheaves  $\E$ may contain no locally free sheaves and  codimension of its singular locus  $\Sing \E$ in the scheme 
$\Sigma=T\times S$ can be greater or equal two. 

We construct a transformation of the initial flat family of coherent torsion-free $\OO_S$-sheaves with the base algebraic scheme $T$, to a flat family of admissible pairs. The base  of the new family of admissible pairs is a scheme $\widetilde T$ which can differ from $T$. 
\subsection{Flatifying base changing}
We start with a flat family of schemes $p\colon \Sigma \to T$ and a sheaf of $\OO_{\Sigma}$-modules $W$ which is flat over $T$. If $\hd W=1,$ then the transformation $\sigma\!\!\!\sigma\colon \widetilde{\Sigma} \to \Sigma$ of the scheme $\Sigma$, which is defined by the sheaf of ideals $\I=\FFitt^0 \EExt^1(W, \OO_{T\times S})$, leads to the scheme $\widetilde{\Sigma}$ of one of two forms:  either $\widetilde{\Sigma} =\Proj \bigoplus_{s\ge 0} \I^s$ or $\widetilde{\Sigma} =\Proj \bigoplus_{s\ge 0} (\I[t]+(t))^s/(t^{s+1})$ depending on the structure of $W$), with a locally free sheaf  $\sigma\!\!\!\sigma^{\ast} W /\tors$. Here $\tors$ is an appropriate subsheaf supported on a proper closed subset of the scheme  $\widetilde{\Sigma}$. Generally, $\widetilde{\Sigma}$ is not obliged to be flat over $T$, as mentioned in  \S \ref{s2}). We are intended to find such a base changing  $\xi \colon \widetilde T \to T$, that it be surjective in the scheme sense (i.~e. the scheme-theoretic image of the morphism $\xi$ coinsides with $T$) and the inverse image of the sheaf $\I$ be flat over $\widetilde T$. 

We need the following result.

Let $T$ be a Noetherian scheme, $X$ a $T$-scheme of finite type, $\MM$ a coherent $\OO_X$-module, $Z \subset T$ a closed subscheme which is defined by the sheaf of ideals $\JJ \subset \OO_T$, $\xi \colon \widehat T \to T$ a morphism of blowing up in the sheaf of ideals $\JJ$. We intriduce the notations: $\widehat \MM=\MM \otimes_{\OO_T}\OO_{\widehat T}$, $\widehat Z=Z\times _T \widehat T$, $Y=Z\times_T X$, $\widehat X=X\times _T \widehat T$, $\widehat Y=Y\times_T \widehat T=\widehat Z\times_{\widehat T} \widehat X$.  Then $\widehat Z$ is a divisor in $\widehat T$. Define a subscheaf  $\widehat \NN \subset \widehat \MM$ as follows: for any affine open subset  $\widehat U \subset \widehat X$ set $\widehat \NN (\widehat U)=\{s\in \widehat \MM (\widehat U)| \Supp s\subset \widehat Y \cap \widehat U \}$.

\begin{definition} \label{puretrans} \cite[Chap. 4, Definition 1]{Ray} {\it The pure transform} $\xi^{\Delta}\MM$  of $\OO_X$-module $\MM$ by the blowing up  $\xi\colon \widehat T \to T$ is a coherent sheaf  $\widehat \MM/\widehat \NN.$  
\end{definition}

\begin{remark} A pure transform is characterized by the following three properties: a) $\xi^{\Delta}\MM$ is a coherent quotient of the inverse image  $\widehat \MM=\xi^\ast \MM$ of the sheaf $\MM$; b) the canonical morphism  $\widehat \MM \to \xi^{\Delta}\MM$ is an isomorphism on the isomorphic open subsets  $\widehat X\setminus \widehat Y \cong X \setminus Y$; c) $\Ass (\xi^{\Delta}\MM)\subset \widehat X\setminus \widehat Y$. 

In particular \cite[Chap. 4]{Ray}, if $\MM$ is  $T$-flat, then $\widehat \MM$ is also $\widehat T$-flat. Since  $\widehat Z$ is a divisor on  $\widehat T$, we have $\Ass (\widehat T)\subset \widehat T\setminus \widehat Z$ and hence  $\Ass (\widehat \MM) \subset \widehat X\setminus \widehat Y$. Consequently, $\widehat \MM=\xi^{\Delta}\MM$.
\end{remark}

\begin{theorem} \label{flatif} \cite[Chap. 4, Theorem 1]{Ray} Let $T, X, \MM$ be as before,  and suppose that $U$ is an open subscheme in $T$ such that the sheaf $\MM|_{X\times _T U}$ is flat over  $U$. Then there exists a blowing up  $\xi\colon \widehat T \to T$ with centre in  $T\setminus U$ such that the pure transform  $\xi^{\Delta}\MM$ of the sheaf $\MM$ is $\widehat T$-flat.
\end{theorem}

Then for the sheaf of ideals  $\I=\FFitt^0 \EExt^1(W, \OO_{\Sigma})$ there is an appropriate blowing up morphism for the base $\xi\colon \widetilde T \to T$ such that the pure transform $\xi^\Delta \I$ is flat over $\widetilde T.$ Then  $\widehat \Sigma:=\Sigma \times_T \widehat T$, and an inclusion $\I \subset \OO_{\Sigma}$ leads to a morphism $\xi^\ast \I \to \OO_{\widehat \Sigma}$, where the sheaf $\OO_{\widehat \Sigma}$ is pure because the scheme  $\widehat \Sigma$ has no embedded components by the restriction imposed on the base $T$ (Convention \ref{noembed}). Hence we have two pure quotients $\xi^\Delta \I$ and $\xi^{-1}\I \cdot \OO_{\widehat \Sigma}$, which are equal  (by the definition of pure quotient). Now $\xi^\Delta \I= \xi^{-1} \I \cdot \OO_{\widehat \Sigma}.$

Next we apply a blowing up $\sigma\!\!\!\sigma \colon \Proj \bigoplus_{s\ge 0} (\xi^{\Delta} \I)^s \to \widehat \Sigma$ or, in dependance on the structure of the family of sheaves under the resolution, the transformation which is analogous to the one in the preceding papers of the author: $\sigma\!\!\! \sigma\colon \Proj \bigoplus_{s\ge 0} (\xi^\Delta \I[t]+t)^s/(t^s) \to \widehat\Sigma$. Both two projective spectra are flat over the base $\widetilde T:=\widehat T$. Hence, for one-step transformation which takes a flat family  $p\colon \Sigma \to T$ to flat family $\widetilde \pi \colon \widetilde \Sigma \to \widetilde T$, we have a commutative diagram with fibred square from the left:
\begin{equation}\label{onestep}
\xymatrix{\Sigma \ar[d]_p &\ar[l]_{\widehat \xi} \widehat \Sigma \ar[d]_\pi &\ar[l]_{\sigma\!\!\!\sigma} \widetilde \Sigma \ar[d]^{\widetilde \pi}\\
T&\ar[l]_\xi \widehat T&\ar@{=}[l] \widetilde T}
\end{equation}

\begin{remark} If homological dimension of the sheaf  $\E$ equals 1, then its transformation is defined by the diagram (\ref{onestep}) for $W=\E$. If homological dimension of the sheaf $\E$ is greater then $1$, i.~e. $\hd \E = \ell >1$, then the resolution requires $\ell$ consequent transformations. Then to preserve flatness of the family of schemes in each step of resolution each  $\sigma\!\!\!\sigma_i$, $i=1, \dots, \ell,$ is preceded by the corresponding blowing up of the base $\xi_i$, $i=1, \dots, \ell.$ Hence,  the transformation of the sheaf $\E$ of  homological dimension  $\ell >1$ is described by the diagram (\ref{onestep}) which is iterated $\ell$ times for the sequence of sheaves  $W$associated with $\E$.
\end{remark}

\subsection{Admissible schemes and nothing else} The notion of an admissible scheme is given by the following 
\begin{definition} For a fixed variety or scheme $S$ the scheme $\widetilde S$ is called {\it admissible} if $\widetilde S \cong S$ or there exist a chain of scheme morphisms  $$\widetilde S=S_\ell \stackrel{\sigma_\ell}{\longrightarrow} S_{\ell-1} \stackrel{\sigma_{\ell-1}}{\longrightarrow} \cdots \stackrel{\sigma_2}{\longrightarrow} S_1 \stackrel{\sigma_1}{\longrightarrow} S_0=S$$
such that $\sigma_i\colon S_i=\Proj_{s\ge 0} (I_i[t]+(t))^s/(t)^{s+1} \to S_{i-1}$ is a structure morphism, $I_i\subset \OO_{S_{i-1}}$ is non-invertible sheaf of ideals, $i=1, \dots, \ell$.
\end{definition}

It remains to confirm that under the base transformation which flatifies the sheaf of ideals  $\I\subset \OO_{\Sigma}$, one gets families which consist of admissible schemes only. It suffices to consider the one-step transformation  (\ref{onestep}) where the form of the morphism  $\sigma\!\!\!\sigma $ depends on the structure of the family  $\E$.

{\it a}) {\it The transformation by a blowing up.} Let the initial family of coherent sheaves $W$ under the resolution, with the base $T$,  contains locally free sheaves. Then the locus $T_0$ of those points in $T$ which correspond to locally free sheaves, is non-empty open subset. Choose an arbitrary closed point $ t \in T\setminus T_0$ and such a curve $C=\Spec k[t] \subset T$ that the image of its point $\{t=0\}$ coincides with $t\in T$, and $C_0:=C\setminus \{t=0\} \subset T_0$. The open subset $C_0 \subset T_0$ does not undergo any changes under base transformation $\xi \colon \widetilde T \to T$. Its closure $\overline C_0$ in $\widetilde T$ is isomorphic to $C$, and we denote it by the same letter: $C \subset \widetilde T$. Also $C \setminus C_0=\{\widetilde t\},$ $\widetilde t \in \widetilde T$, and $\xi (\widetilde t)=t$. 

Blowing up of the scheme  $\widehat \Sigma:=\Sigma\times_T \widetilde T$ in the sheaf of ideals $\xi^{\Delta}\I$ we arrive to a scheme  $\widetilde \Sigma$ which is flat over  $\widetilde T$. It carries an invertible sheaf  $\widetilde \L$ which is very ample over $\widetilde T$. Also Hilbert polynomials of fibres of the morphism $\widetilde \pi\colon \widetilde \Sigma \to \widetilde T$ are equal if they are computed with respect to $\widetilde \L$. For any closed point $t\in \widetilde T$ consider a projective space  $\P=P(H^0(\widetilde \pi^{-1}(t),\widetilde \L^n|_{\widetilde \pi^{-1}(t)})$ and an immersion $j_t\colon \widetilde \pi^{-1}(t) \hookrightarrow \P$ which is induced by the invertible sheaf  $\widetilde \L^n|_{\widetilde \pi^{-1}(t)}$. Let  $P(n)=\xi(j_t^\ast \OO_\P(n))$. Since the scheme  $\widehat \Sigma$ is flat over $\widetilde T$, the polynomial $P(n)$ does not depend on the choice of  $t$.

Now consider the Hilbert scheme  $\Hilb^{P(n)}\P$ of subschemes in  $\P$ with Hilbert polynomial  $P(n)$, and the universal subscheme $\Univ^{P(n)}\P \subset (\Hilb^{P(n)}\P) \times \P.$
Then there are morphisms $\widetilde \mu, \mu$ included into the fibred square
$$\xymatrix{\widetilde \Sigma \ar[d]_{\widetilde \pi} \ar[r]^<<<<<{\widetilde \mu}& \Univ^{P(n)}\P \ar[d]\\
\widetilde T \ar[r]^<<<<<\mu &\Hilb^{P(n)}\P}
$$
In particular, the fibre  $\widetilde S_t:=\widetilde \pi^{-1}(\widetilde t)$ over the point $\widetilde t\in \widetilde T$ is immersed into the space  $\P$ by the morphism  $j_{\widetilde t} \colon \widetilde \pi^{-1}(\widetilde t) \hookrightarrow \P$. This morphism is defined by the ample invertible sheaf  $\widetilde L^n_t:=\widetilde \L^n|_{\widetilde \pi^{-1}(\widetilde t)}$.
If $\widetilde t\in \widetilde T_0=T_0$, then $\widetilde S_{\widetilde t}\cong S$ and $\widetilde L_{\widetilde t} \cong L$.

Note that the fibre  $\widetilde S_t:=\widetilde \pi^{-1}(t)$ contains the component of admissible scheme which is obtained by blowing up of the variety $S$ in the ideal $I_t:=(\xi^\Delta \I)|_{\widetilde \pi^{-1}(\widetilde t)}=\I|_{t\times S}$, as its irreducible component.

Consider such an immersion  $C \hookrightarrow \widetilde T$ that $C\setminus 0 \subset T_0$ and the point  $0\in C$ corresponds to a coherent sheaf which is not locally free. Set $\Sigma_C:=\Sigma \times_T C$ and let  $W_C=W|_{\Sigma_C}$ be a restriction of the family of coherent sheaves to $\Sigma_C$. If the resolution of the family of coherent sheaves with a base  $\widetilde T$ includes the blowing up in the sheaf of ideals  $\xi^\Delta \I$, then the procedure of resolution of obtained family with a base  $C$ includes the blowing up in the sheaf of ideals  $I_t[t]+(t) \subset \OO_{C \times T}$. The fibres of both two families at the point $t=0\in \widetilde T\setminus  T_0$ contain the component which is isomorphic to the variety obtained by blowing up of $S$ in the sheaf of ideals  $I_t$. Fibres at points in  $C\cap T_0$ are isomorphic to $S$. Denote by  $\pi_C \colon \widetilde \Sigma_C \to C$ the family which arises from the resolution of the family with the base $C$. 
Then we have a commutative diagram with fibred square from the right: 
$$\xymatrix{\widetilde \Sigma_C \ar[d]_{\pi_C} \ar@{^(->}[r]&\widetilde \Sigma \ar[d]_{\widetilde \pi} \ar[r]^<<<<<{\widetilde \mu}& \Univ^{P(n)}\P \ar[d]\\
C\ar@{^(->}[r] &\widetilde T \ar[r]^<<<<<\mu &\Hilb^{P(n)}\P}
$$
Fibres $\pi_C^{-1}(0) \subset \widetilde \pi^{-1}(0)$ have equal Hilbert polynomials. They are computed with respect to $\widetilde L_t=\widetilde \L|_{\widetilde \pi^{-1}(t)}$ and equal $P(n)$. This implies that these fibres coincide.

{\it b}) {\it The transformation by the base enhancing.} Now consider the case when the initial family of coherent sheaves $W$ with a base  $T$ contains no locally free sheaves. In this case the transformation of the form $\sigma\!\!\!\sigma \colon \widetilde \Sigma \to \widehat \Sigma$, where $\widetilde \Sigma=\Proj \bigoplus_{s\ge 0} (\xi^\Delta \I [t]+(t))^s/(t^{s+1})$, is done instead of the blowing up of the family $\pi \colon \widehat \Sigma \to \widetilde T$. Since the sheaf of ideals $\xi^\Delta \I$ is flat $\OO_{\widetilde T}$-module, then $\widetilde \Sigma$ is flat over   $\widetilde T$. 

Prove that in this case the only admissible schemes arise as fibres of the new family. For this purpose we construct an immersion  $\widetilde T \cong \widetilde T \times \{ 0\} \hookrightarrow \widetilde T \times C=:T'$, where $C=\Spec k[t]$ and consider $T'\times S$ as the initial family.

For our purpose it suffices to take one-step blowing up $\widetilde{T' \times S} \to T'\times S$ in the sheaf of ideals  $\xi^\Delta \I [t]+(t).$ The scheme $\widetilde{T' \times S}$ is flat over  $T'$. For any point  $\widetilde t \in \widetilde T$ consider a subscheme $\widetilde t \times C \subset T'$ and the corresponding subfamily  $\widetilde \Sigma_{\widetilde t}:=\widetilde t \times _{\widetilde T}\widetilde{T'\times S} \to \widetilde t \times C$.  
It is a flat family and  $\widetilde \Sigma_{\widetilde t}=\Proj \bigoplus_{s\ge 0} (\xi^\Delta \I[t]+(t))^s/(t)^{s+1}$. Now consider the sheaf of ideals  $\xi^\Delta \I|_{\widetilde t \times S}=:I_{\widetilde t}$ and the blowing up of the family  $\widetilde t \times C \times S$ in the sheaf of ideals  $I_{\widetilde t}[t]+(t)$, i.~e. $\sigma\!\!\!\sigma\colon \widehat{\widetilde t \times C \times S} \to \widetilde t \times C \times S$. The scheme  $\widehat{\widetilde t \times C \times S}$ is also flat over  $\widetilde t \times C$, and all its fibres over points  $(\widetilde t, t)$ for $t\ne 0$, are isomorphic to $S$, as well as the corresponding fibres of the scheme $\widetilde \Sigma_{\widetilde t}$. 
The family $\widetilde{T' \times S}$ has a  very ample (relative to $T'$) invertible sheaf  $\widetilde \L=\OO_{T'} \boxtimes L^n \otimes {\sigma\!\!\! \sigma} ^{-1}(\xi^\Delta \I) \cdot \OO_{\widetilde{T' \times S}}$, and Hilbert polynomials of fibres of this family, if being compute with respect to  $\widetilde \L$, are equal for all the fibres. The family  $\widetilde \Sigma_{\widetilde t}$ carries very ample (relative to $\widetilde t \times C$) invertible sheaf  $\widetilde \L_{\widetilde t}=k_{\widetilde t} \boxtimes \OO_{\A^1} \boxtimes L^n \otimes \sigma\!\!\!\sigma_{\widetilde t}^{-1} I_t \cdot \OO_{\widetilde \Sigma_{\widetilde t}}$, which also provides equal Hilbert polynomial on all fibres. The invertible sheaves are related as $\widetilde \L_{\widetilde t}=\widetilde \L |_{\widetilde \Sigma_{\widetilde t}}$. By the universal property of blowups, the fibre of the family $\widehat{\widetilde t \times C \times S}$ over a point $(\widetilde t,0)$, being an admissible scheme, is a sibscheme of the fibre of the family  $\widetilde{T'\times S}$ over the same point.  By the equality of Hilbert polynomials for fibres of both families at the point $(\widetilde t, 0)$ it implies that these fibres coincide as schemes.

\subsection{The procedure of resolution for a family of sheaves} Let $T$ be an arbitrary (possibly non-reduced) $k$-scheme of finite type. Assume that its reduction $T_{red}$ is irreducible.
If $S$ is a surface and  $\E$ is a family of coherent torsion-free sheaves on $S$ over the base $T$ then the homological dimension of $\E$ as an $\OO_{T\times S}$-module is not greater then~1. In our case $S$ has the higher dimension and we have to work with a locally free resolution of the higher length. Set $\Sigma=T\times S$. Start with a shortest locally free resolution of the family of sheaves $\E$ and cut the corresponding exact $\OO_{\Sigma}$-sequence of  length $\ell$
\begin{equation}
  0\to \widehat E_\ell \to \widehat E_{\ell-1}\to \dots \to \widehat E_0 \to \E \to 0 \label{resol}
\end{equation} with locally free $\OO_{\Sigma}$-modules $\widehat E_\ell, \dots \widehat E_0$ into triples:
\begin{equation}
0\to W_i \to \widehat E_{i-1} \to W_{i-1} \to 0. \label{triple}
\end{equation}
Here $W_\ell=\widehat E_\ell$, $W_1=\ker(\widehat E_0 \to \E)$ and $W_i=\ker(\widehat E_{i-1} \to \widehat E_{i-2})=\coker (\widehat E_{i+1}\to \widehat E_i)$ for $i=2,\dots, \ell-1.$ Also we keep in mind that all the sheaves $W_i$ except $W_\ell$ are not locally free (otherwise the resolution can be shorter). Since $S$ is assumed to be regular,  $E$ possesses a locally free resolution of  length not greater then $\dim S$.

Singularities of the sheaf $\E$ can be resolved by the consequent transformations  $\sigma\!\!\!\sigma_i \colon \widehat \Sigma_i \to \Sigma_i$, alternating with flatifying transformations $\widehat \xi_i\colon \Sigma_i \to \widehat \Sigma_{i-1}$, $i=1, \dots, \ell,$ for the  sheaves of ideals
\begin{eqnarray} \I_1=\FFitt^0 \EExt ^1(\widehat \xi_1^\ast  W_{\ell-1}, \OO_{\Sigma_1})\subset \OO_{\Sigma_1},\nonumber \\ \I_2=\FFitt^0 \EExt^1(\widehat \xi_2^\ast \sigma\!\!\!\sigma_1^\ast
\widehat \xi_1^\ast W_{\ell-2},\OO_{\Sigma_2})\subset \OO_{\Sigma_2}, \nonumber \\ \I_3=\FFitt^0
\EExt^1(\widehat\xi_3^\ast\sigma\!\!\!\sigma_2^\ast \widehat \xi_2^\ast\sigma\!\!\!\sigma_1^\ast \widehat \xi_1^\ast W_{\ell-3},\OO_{\Sigma_3})\subset \OO_{\Sigma_3}, \nonumber \\.\;.\;.\;.\;.\;.\;.\;.\;.\;.\;.\;.\;.\;.\;.\;.\;.\;.\;.\;.\;.\;.\;.\;.\;.\;.\;.\;.\;.\;.\;.\;.\;.\;.\;.\;.\;,\nonumber \\ \I_{\ell}=\FFitt^0 \EExt^0(\widehat \xi_\ell^\ast\sigma\!\!\!\sigma_{\ell-1}^\ast \dots \widehat \xi_2^\ast
\sigma\!\!\!\sigma_1^\ast \widehat \xi_1^\ast W_0, \OO_{\Sigma_\ell})\subset \OO_{\Sigma_\ell}.\label{idsh}
\end{eqnarray} These morphisms are included into the following commutative diagram  
\begin{equation*} \xymatrix{
\widehat \Sigma_{\ell} \ar[d]_{\widehat \pi_\ell}
\ar[r]^{\sigma\!\!\!\sigma_\ell}&\Sigma_\ell \ar[d]_{\pi_\ell} \ar[r]^{\widehat \xi_{\ell}}&\widehat \Sigma_{\ell-1} \ar[d]_{\widehat \pi_{\ell-1}}\ar[r]^{\sigma\!\!\!\sigma_{\ell-1}}&\dots
\ar[r]^{\widehat \xi_2}\ar[r]& \widehat \Sigma_1 \ar[d]_{\widehat \pi_1}\ar[r]^{\sigma\!\!\!\sigma_1}&\Sigma_1 \ar[d]_{\pi_1} \ar[r]^{\widehat \xi_1}&\widehat \Sigma_0 \ar[d]_{\widehat \pi_0}\\
T_\ell \ar@{=}[r]&T_\ell \ar[r]^{\xi_\ell}& T_{\ell-1}\ar@{=}[r]&\dots \ar[r]&T_1\ar@{=}[r]&T_1\ar[r]&T}
\end{equation*}
where $\widehat \Sigma_0=\Sigma,$ $T_0=T,$ $\widehat \pi_0=p\colon \Sigma \to T$ is the initial family of schemes (with fibres, isomorphic to  $S$), $\widehat \Sigma_\ell =\widetilde \Sigma,$ $T_\ell=\widetilde T,$  $\widetilde \pi=\widehat \pi_\ell\colon \widetilde \Sigma \to \widetilde T$ is the final family of schemes.

Form a composite morphism $\xi:=\xi_1 \circ \dots \xi_\ell \colon \widetilde T \to T$.  If the sheaves of ideals  (\ref{idsh}) are flat over the corresponding base schemes $T_i,$ $i=1, \dots, \ell$, then as well the sheaves of ideals  $\xi_\ell^\Delta \dots \xi_{i+1}^\Delta \I_i \subset \OO_{\Sigma\times_T \widetilde T}$, $i=1, \dots, \ell$, are also flat over   $\widetilde T.$

\begin{convention}
We assume in the sequel that the base change  $\xi \colon \widetilde T \to T$ is done initially before the transformations given by the morphisms  $\sigma\!\!\!\sigma_i$, $i=1, \dots, \ell$. To simplify the notation we donote the scheme  $\Sigma \times_T \widetilde T$ by $\Sigma=\Sigma_0$, its projection to the base  $\widetilde T$ by $p \colon \Sigma \to \widetilde T$, the family of scheaves  $\widehat \xi^\ast \E$ by $\E$, and family of very ample invertible sheaves  $\widehat \xi^\ast \L$ by $\L$. The transformations  $\sigma\!\!\!\sigma_i \colon \Sigma_i \to \Sigma_{i-1}$ are defined by the sheaves of ideals 
\begin{eqnarray} \I_1=\FFitt^0 \EExt ^1(  W_{\ell-1}, \OO_{\Sigma_0})\subset \OO_{\Sigma_0},\nonumber \\ \I_2=\FFitt^0 \EExt^1(\sigma\!\!\!\sigma_1^\ast
  W_{\ell-2},\OO_{\Sigma_1})\subset \OO_{\Sigma_1}, \nonumber \\ \I_3=\FFitt^0
\EExt^1(\sigma\!\!\!\sigma_2^\ast \sigma\!\!\!\sigma_1^\ast  W_{\ell-3},\OO_{\Sigma_2})\subset \OO_{\Sigma_2}, \nonumber \\.\;.\;.\;.\;.\;.\;.\;.\;.\;.\;.\;.\;.\;.\;.\;.\;.\;.\;.\;.\;.\;.\;.\;.\;.\;.\;.\;.\;.\;.\;.\;.\;.\;,\nonumber \\ \I_{\ell}=\FFitt^0 \EExt^0(\sigma\!\!\!\sigma_{\ell-1}^\ast \dots
\sigma\!\!\!\sigma_1^\ast  W_0, \OO_{\Sigma_{\ell-1}})\subset \OO_{\Sigma_{\ell-1}}.\label{idsh1}
\end{eqnarray}
\end{convention}

Each segment  $$0\to W_{i} \to \widehat E_{i-1} \to W_{i-1} \to 0$$ is resolved by the morphism  $\sigma\!\!\!\sigma_{\ell-i+1}$ on the 
$(\ell-i+1)$-th step of the process of resolution. The form of the morphism  $\sigma\!\!\!\sigma_{\ell-i+1}$ is determined by presence or absence of locally free sheaves among the restrictions of the sheaf  $\sigma\!\!\!\sigma_{\ell-i}^\ast \dots \sigma\!\!\!\sigma_1^\ast W_{i-1}$to fibres of the family  $\pi_{\ell-i}\colon \Sigma_{\ell-i} \to T_{\ell-i}.$
On each step we need the following lemma.


\begin{lemma}\label{mtlemma}\cite[Lemma 2]{Tim8} Let  $X$ be a Noetherian scheme
such that its reduction $X_{\red}$ is irreduc\-ible, $\FF$ nonzero
coherent $\OO_X$-sheaf supported on a subscheme of codimension $\ge
1$. Then the sheaf of 0-th Fitting ideals $\FFitt^0(\FF)$ is an
invertible $\OO_X$-sheaf if and only if $\FF$ has homological
dimension equal to 1: $\hd_X \FF=1.$
\end{lemma}


\begin{remark} If the scheme $ X=\Sigma_i$ has an irreducible
reduction, then this lemma is applicable immediately and we conclude
that  $$\hd \sigma \!\!\! \sigma_{i+1} ^{\ast}\EExt^1(\sigma \!\!\! \sigma_{i}^\ast \dots \sigma \!\!\! \sigma_1 W_{\ell-i},
\OO_{\Sigma_{i-1}})=1.$$ If $ \Sigma_i$ has a reducible reduction then there is
a natural decomposition $\Sigma_i = \Sigma_i^0 \cup \D_i'$ where
$\Sigma_i^0 \cong \Proj \bigoplus _{s\ge 0} \I_i^s$ is isomorphic to the scheme
obtained by  blowing up $\Sigma_{i-1}$ in the sheaf of ideals $\I_i$ and
$\D_i'$ is an exceptional divisor of the blowing up morphism $\sigma
\!\!\! \sigma_i'\colon \Sigma'_i \to \Sigma_{i-1} \times \A^1.$ The
scheme-theoretic intersection $\Sigma_i^0$ and $\D_i'$ equals the exceptional divisor $\D_i$
of the blowing up morphism $\sigma \!\!\! \sigma_i^0\colon \Sigma_i^0 \to
\Sigma_{i-1}$ defined by the sheaf of ideals  $I_i$:
$$\Sigma_i^0 \cap \D_i'=\D_i.$$ We come to the decomposition of the morphism $\sigma\!\!\! \sigma_i\colon
\Sigma_i \stackrel{\delta \!\!\delta_i}{\longrightarrow} \Sigma_i^0
\stackrel{\sigma \!\!\! \sigma_i^0}{\longrightarrow} \Sigma_{i-1}$,
where $\delta \!\!\delta_i$ is an identity on $ \Sigma_i^0$. 
Its action on $\D_1'$
$$
\delta \!\!\delta_1 |_{\D_1'}\colon \D_1' \to  \Sigma_1^0
$$
factors through the exceptional divisor $\D_i=\Proj \bigoplus_{s\ge
0}\I_i^s/\I_i^{s+1}$ of the morph\-ism $\sigma \!\!\! \sigma_i^0$
and is defined by the structure of $\bigoplus_{s\ge 0}
\I_i^s/\I_i^{s+1}$-algebra on the graded ring $\bigoplus_{s\ge 0}
\I_i'^s/\I_i'^{s+1}.$
\end{remark}

\begin{convention}
For uniformity of notation we use the decomposition $\sigma\!\!\!\sigma_i = \sigma\!\!\!\sigma^0_i \circ \delta\!\!\delta_i$ for all $i=1,\dots,\ell$ with no reference  if $\delta\!\!\!\,\delta_i$ is either a non-identity or an identity morphism.
\end{convention}

If on the $i$-th step the transformation had the form  $\Sigma_i=\Proj \bigoplus_{s\ge 0} (\I_i[t]+(t))^s/(t^{s+1}),$ then Lemma \ref{mtlemma} is not applicable at the  $(i+1)$-th step.

To overcome this obstacle we return to the decomposition
$\sigma\!\!\!\sigma_i=
\sigma\!\!\!\sigma_i^0 \circ \delta\!\!\delta_i$ where $\sigma\!\!\!\sigma_i^0$ is a blowup
morphism.

We are interested in the composite morphisms $
\delta\!\!\delta_{i} \circ \sigma\!\!\!\sigma_{i+1}^0 $. The idea is to ``interchange'' the morphisms
in this composite by the commutative diagram
\begin{eqnarray}\label{interch}\xymatrix{
\Sigma_{i+1}^0 \ar[d]_{\delta\!\!\delta_i^{-1}}
\ar[r]^{\sigma\!\!\!\sigma_{i+1}^0}&
\Sigma_i \ar[d]^{\delta\!\!\delta_i}\\
\Sigma_{i+1}^{-1}\ar[r]^{\sigma\!\!\!\sigma_{i+1}^{-1}}& \Sigma_i^0
}\end{eqnarray} as it is explained below.
\begin{remark} We use double indexing with lower and upper indices for
schemes and morph\-isms. For example, $-1$ in notation
$\delta\!\!\delta_i^{-1}$ or $\sigma\!\!\!\sigma_{i+1}^{-1}$ is just
index but neither a sign of an inverse of the map nor a notation for
an inverse image of a sheaf.
\end{remark}
Since $\Sigma_{i+1}^0$ is obtained by a blowing up of the scheme
$\Sigma_i$ in the sheaf of ideals $\I_{i+1}\subset \OO_{\Sigma_i}$
and it can be interpret\-ed as a projective spectrum
$\Sigma_{i+1}^0=\Proj A_{i+1}^0$ of the graded $\OO_{\Sigma_i}$-algebra
$A_{i+1}^0=\bigoplus_{s_i\ge 0}\I_{i+1}^{s_i}$. But $\Sigma_i$
itself is a projective spectrum $\Sigma_i=\Proj A_i$ of the~$\OO_{\Sigma_{i-1}}$-algebra $A_i=\bigoplus_{s_i\ge 0}
(\I_i[t_i]+(t_i))^{s_i}/(t_i)^{s_i+1}$ for $\I_{i} \subset
\OO_{\Sigma_{i-1}}$, $\Sigma_i^0=\Proj A_i^0$ for the~$\OO_{\Sigma_{i-1}}$-algebra $A_i^0=\bigoplus _{s_i\ge 0}
\I_i^{s_i}$, and the morphism $\delta\!\!\delta_i$ is induced by the structure of an~$A_i^0=\bigoplus_{s_i\ge 0} \I_i^{s_i}$-algebra on $A_i=\bigoplus_{s_i\ge 0} (\I_i[t_i]+(t_i))^s_i/(t_i)^{s_i+1}$.

Obviously, $A_i^0=\bigoplus_{s_i\ge 0} \I_i^{s_i}$ is a subalgebra in
$A_i$; it is formed by and only by those elements that have degree
zero in $t_i$, i.e. $A_i^0=A_i|_{t_i=0}.$

Define $\I_{i+1}^0:=\I_{i+1} \cap A_i^0$ as an
intersection with the subalgebra $A_i^0$ in $A_i$ of elements of zero degree in $t_i$.

Obviously, $(\I_{i+1}^0)^u=(\I_{i+1})^u \cap A_i^0$ (whereas $\I_{i+1}$ consists of polynomials in $t_i$, the left hand side and the right hand side of the latter equality consist of homogeneous terms of zero degree in $t_i$) for any $u\ge 0$. 
Define
$$A_{i+1}^{-1}:=\bigoplus_{s_{i+1}\ge 0} \bigoplus_{s_i\ge
0}\I_i^{s_i}(\I_{i+1}^0)^{s_{i+1}}=A_{i+1}^0|_{t_i=0}$$ as a subalgebra of $A_{i+1}^0$ formed by elements of zero degree  in $t_i$. The corresponding scheme is $\Sigma_{i+1}^{-1}:=\Proj A_{i+1}^{-1}$.
The grading on $A_{i+1}^{-1}$ which fixes homogeneous elements when the projective spectrum is formed, is defined by the inclusion to $A_{i+1}^0$ and by the degrees $s_{i+1}$. The morphism $\sigma\!\!\!\sigma_{i+1}^{-1}\colon \Sigma_{i+1}^{-1} \to \Sigma_i^0$ which induced by the inclusion $A_i^0 \hookrightarrow A_{i+1}^{-1}$ is the morphism of the bolwing up in the sheaf of ideals  $\I_{i+1}^0$.
By the construction $A_{i+1}^0=A_{i+1}^{-1} \otimes_{A_i^0} A_i$, and the
cor\-responding square of schemes (\ref{interch}) is fibred.


By iterating of the square similar to (\ref{interch}) the composite of the morphisms $\sigma\!\!\!\sigma_i=\delta\!\!\delta_i \circ
\sigma\!\!\!\sigma_i^0$, $i=1,\dots,\ell,$ can be
decomposed into the following diagram:
\begin{equation}\xymatrix{
\Sigma_\ell \ar[d]_{\delta\!\!\delta^0_\ell}\ar[dr]^{\sigma\!\!\!\sigma_\ell}\\
\Sigma_\ell^0 \ar@{.}[d]\ar[r]^{\sigma\!\!\!\sigma_\ell^0}
&\Sigma_{\ell-1}
\ar@{.}[d] \ar@{.}[rd]\\
\Sigma_\ell^{-i+2} \ar[d]_{\delta\!\!\delta_{i+1}^{-i+1}}
\ar[r]^{\sigma\!\!\!\sigma_{\ell}^{-i+2}} &\Sigma_{\ell-1}^{-i+3}
\ar[d]_{\delta\!\!\delta_{i+1}^{-i+2}} \ar@{.}[r] &\Sigma_{i+1}
\ar[d]_{\delta\!\!\delta_{i+1}^0}
\ar[rd]^{\sigma\!\!\!\sigma_{i+1}}\\
\Sigma_\ell^{-i+1}
\ar[d]_{\delta\!\!\delta_i^{-i}}\ar[r]^{\sigma\!\!\!\sigma_\ell^{-i+1}}
&\Sigma_{\ell-1}^{-i+1} \ar[d]_{\delta\!\!\delta_i^{-i+1}}\ar@{.}[r]
&\Sigma_{i+1}^0 \ar[d]_{\delta\!\!\delta_i^{-1}}
\ar[r]^{\sigma\!\!\!\sigma_{i+1}^0}&
\Sigma_i \ar[d]_{\delta\!\!\delta_i^0} \ar[rd]^{\sigma\!\!\!\sigma_i}\\
\Sigma_\ell^{-i}
\ar@{.}[d]\ar[r]^{\sigma\!\!\!\sigma_\ell^{-i}}&\Sigma_{\ell-1}^{-i+1}
\ar@{.}[d]\ar@{.}[r]&\Sigma_{i+1}^{-1}
\ar[r]^{\sigma\!\!\!\sigma_{i+1}^{-1}}\ar@{.}[d] &\Sigma_i^0
\ar[r]^{\sigma\!\!\!\sigma_i^0}\ar@{.}[d]
&\Sigma_{i-1} \ar@{.}[d] \ar@{.}[rd]\\
\Sigma_{\ell}^{ -\ell+2} \ar[d]_{\delta\!\!\delta_1^{-\ell+1}}
\ar[r]^{\sigma\!\!\!\sigma_\ell^{-\ell+2}}&\Sigma_{\ell-1}^{-\ell+3}
\ar[d]_{\delta\!\!\delta_1^{-\ell+2}}\ar@{.}[r]&\Sigma_{i+1}^{-i+1}
\ar[d]_{\delta\!\!\delta_1^{-i}}\ar[r]^{\sigma\!\!\!\sigma_{i+1}^{-i+1}}&\Sigma_i^{-i+2}
\ar[d]_{\delta\!\!\delta_1^{-i+1}}
\ar[r]^{\sigma\!\!\!\sigma_i^{-i+2}}&\Sigma_{i-1}^{-i+3}
\ar[d]_{\delta\!\!\delta_1^{-i+2}} \ar@{.}[r]
&\Sigma_1 \ar[d]_{\delta\!\!\delta_1^0} \ar[rd]^{\sigma\!\!\!\sigma_1}\\
\Sigma_{\ell}^{-\ell+1}
\ar[r]^{\sigma\!\!\!\sigma_\ell^{-\ell+1}}&\Sigma_{\ell-1}^{-\ell+2}
\ar@{.}[r]&\Sigma_{i+1}^{-i}
\ar[r]^{\sigma\!\!\!\sigma_{i+1}^{-i}}&\Sigma_i^{-i+1}
\ar[r]^{\sigma\!\!\!\sigma_i^{-i+1}}&\Sigma_{i-1}^{-i+2}
\ar@{.}[r]&\Sigma_1^0
\ar[r]^{\sigma\!\!\!\sigma_1^0}&\Sigma}\label{bigdiag}
\end{equation}
Each square of this diagram has a view \begin{equation*}
\xymatrix{\Sigma_i^{-j} \ar[d]_{\delta\!\!\delta_{i-1}^{-j-1}}
\ar[r]^{\sigma\!\!\!\sigma_i^{-j}}& \Sigma_{i-1}^{-j+1}
\ar[d]^{\delta \!\!\delta_{i-1}^{-j}}\\
\Sigma_i^{-j-1} \ar[r]_{\sigma
\!\!\!\sigma_i^{-j-1}}&\Sigma_{i-1}^{-j}}
\end{equation*}
for $i=0, \dots, \ell,$ $j=-1, \dots,i-1$, where $\Sigma_0:=\Sigma$,
$\Sigma_i^1:=\Sigma_i$, $\delta\!\!\delta_i^0:=\delta\!\!\delta_i$.
Index rules are as follows:
\begin{itemize}
\item{for $\Sigma$'s a lower index is constant along the columns and
decreases in the direct\-ion of the horizontal arrows. An upper index is
constant along the skew lines going parallel to the diagonal and
decreases in the direction of the vertical arrows;}
\item{for $\delta\!\!\delta$'s a lower index is constant along the rows
and decreases in the direction of the vertical arrows. An upper index is
constant along the lines going parallel to the diagonal and decreases in
the direction of the vertical arrows;}
\item{for $\sigma\!\!\!\sigma$'s a lower index is constant along the columns
and decreases in the direct\-ion of the horizontal arrows. An upper index is
constant along the lines going parallel to the diagonal and increases in
the direction of the horizontal arrows.}
\end{itemize}

In this diagram the
bottom horizontal row is a composite of consequent blowups.  
The left vertical column $\delta\!\!\delta_1^{-\ell+1} \circ \dots \circ
\delta\!\!\delta_\ell^0$ provides flatness of the scheme
$\Sigma_\ell=\widetilde \Sigma$ over its base~$\widetilde T$ by consequent
``growings up'' additional components of several levels. Application of the functor
$(\delta\!\!\delta_1^{-\ell+1} \circ \dots \circ
\delta\!\!\delta_\ell^0)^\ast$ leads, as we will describe below, to an inverse image of
a locally free sheaf produced from the initial $\OO_\Sigma$-sheaf $\E$ by the morphisms of the bottom horizontal row. Each square of the diagram is an analog of
(\ref{interch}). Also it is useful to keep in mind that each $\Sigma_i^{-i+1}$ is included
into $\Sigma_i$ as a union of some connectad components ($i=1,\dots,\ell$).
\begin{example} Let $\ell =2$ and (\ref{bigdiag}) consists of one
square and two triangles:
\begin{equation}\label{doubletriangle}\xymatrix{\Sigma_2 \ar[d]_{\delta\!\!\!\,\delta_2^0}
\ar[rd]^{\sigma\!\!\!\sigma_2}\\
\Sigma_2^0 \ar[d]_{\delta\!\!\!\,\delta_1^{-1}} \ar[r]^{\sigma\!\!\!\sigma_2^0}&
\Sigma_1 \ar[d]_{\delta\!\!\!\,\delta_1^0} \ar[rd]^{\sigma\!\!\!\sigma_1}\\
\Sigma_2^{-1} \ar[r]^{\sigma\!\!\!\sigma_2^{-1}}& \Sigma_1^0
\ar[r]^{\sigma\!\!\!\sigma_1^0}&\Sigma}
\end{equation}
We do compupations similarly to  \cite{Tim10}. Start with the exact $\OO_{\Sigma}$-sequence $$0\to \widehat E_2 \to \widehat E_{1} \to \widehat E_0 \to \E \to 0,
$$ where sheaves  $\widehat E_i$ are locally free for $i=0,1,2$, and with the exact triple 
\begin{equation}\label{lfr}
0\to \widehat E_2 \to \widehat E_{1} \to W_{1} \to 0.
\end{equation}
produced by cutting the sequence.
Since  $\hd W_{1} = 1$, this sheaf can be resolved in the same fashion as described in  
\cite{Tim10}, by the morphism $\sigma\!\!\!\sigma_1\colon \Sigma_1 \to
\Sigma_0$. Introduce the notation $$H_1:=\ker (\widehat E_2^\vee \to \EExt^1(W_{1},
\OO_{\Sigma_0}))=\coker(W_{1}^{\vee} \to \widehat E_{1}^{\vee})$$ in the dual sequence for (\ref{lfr}):
$$
0\to W_{1}^{\vee} \to \widehat E_{1}^{\vee} \to \widehat E_2^{\vee} \to
\EExt^1(W_{1}, \OO_{\Sigma_0}) \to 0.
$$
Apply an inverse image  $\sigma \!\!\! \sigma_1^{\ast}$:
\begin{eqnarray}\xymatrix{
\sigma \!\!\! \sigma_1^{\ast}\widehat W_1^{\vee}\ar[r]&\sigma \!\!\!
\sigma_1^{\ast}\widehat E_{1}^{\vee}\ar[r]
&\sigma \!\!\! \sigma_1^{\ast}H_1 \ar[r]&0, }  \nonumber\\
\label{c}\xymatrix{ \sigma \!\!\! \sigma_1^{\ast}H_1\ar[r]&\sigma
\!\!\! \sigma_1^{\ast}\widehat E_2^{\vee}\ar[r] &\sigma \!\!\!
\sigma_1^{\ast}\EExt^1(W_{1}, \OO_{\Sigma_0})\ar[r]&0.}
\end{eqnarray}
In (\ref{c}) we denote $N_1:=\ker(\sigma \!\!\!
\sigma_1^{\ast}\widehat E_2^{\vee}\to \sigma \!\!\!
\sigma_1^{\ast}\EExt^1(W_{1}, \OO_{\Sigma_0}))$. The sheaf 
 $\FFitt^0(\sigma \!\!\!
\sigma_1^{\ast}\EExt^1(W_{1}, \OO_{\Sigma_0}))$ is invertible by the functorial property  of Fitting ideals  $\FFitt$: \begin{eqnarray}\FFitt ^0\!(\sigma \!\!\! \sigma_1
^{\ast}\EExt^1\!(W_{1}, \!\OO_{\Sigma_0}\!))\!=\!((\sigma \!\!\!
\sigma_1)^{-1}\!\FFitt ^0\!(\EExt^1\!(W_{1}, \!\OO_{\Sigma_0}\!)))\!\cdot \!
\OO_{\Sigma_1}\!\!=\!\!((\sigma \!\!\! \sigma_1)^{-1}\I_1\!)\!\cdot\!
\OO_{\Sigma_1}.\label{invtty}\end{eqnarray}
In the case when  $\sigma\!\!\!\sigma_1$ is the morphism of the blowing up the chain (\ref{invtty}) proves the invertibility. In the case when $\sigma\!\!\!\sigma_1$ is a transformation with base enhancing we have $\sigma\!\!\! \sigma_1=\sigma\!\!\!\sigma'_1|_{\Sigma_1},$ where $\sigma\!\!\!\sigma'_1 \colon \widehat{\Sigma_0 \times \A^1}\to \Sigma_0 \times \A^1$ is the blowing up in  $\I'_1=\I_1[t]+(t),$ and the immersion  $\Sigma_1 \hookrightarrow \widehat{\Sigma_0 \times \A^1}$ is induced by the fibred square  \begin{equation*}
\xymatrix{\widehat{\Sigma_0 \times \A^1} \ar[r]^{\sigma\!\!\!\sigma_1'}& \Sigma_0
\times \A^1\\
\Sigma_1 \ar@{^(->}[u]^{\widetilde i_0}
\ar[r]^{\sigma\!\!\!\sigma_1}&\Sigma_0 \ar@{^(->}[u]_{i_0}}
\end{equation*}
The immersion $i_0$ identifies  $\Sigma_0$ with the zero fibre  $\Sigma_0 \times 0$. Then
\begin{eqnarray*}
(\sigma \!\!\! \sigma_1^{-1}\I_1)\cdot
\OO_{\Sigma_1}=({\sigma\!\!\!\sigma'_1}^{-1} i_0^{-1}\I_1')\cdot \OO_{
\Sigma_1}=(\widetilde i_0^{-1} {\sigma\!\!\!
\sigma_1'}^{-1}\I_1')\cdot \OO_{\Sigma_1}=\widetilde i_0^{\ast}
({\sigma\!\!\!\sigma_1'}^{-1}\I_{1}' \cdot \OO_{ \Sigma_1'}),
\end{eqnarray*}
where the sheaf of ideals in brackets is invertible. Hence it is invertible also on the component  $\Sigma_1^0 \subset \Sigma_1$. We rewrite it using the decomposition of the morphism $\sigma\!\!\!\sigma_1=\sigma\!\!\!\sigma_1^0 \circ \delta\!\!\!\,\delta_1$:
\begin{eqnarray*}
(\sigma \!\!\! \sigma_1^{-1}\I_1)\cdot
\OO_{\Sigma_1}=({\sigma\!\!\!\sigma_1^0 \circ \delta\!\!\!\,\delta_1}^{-1} \I_1)\cdot \OO_{
\Sigma_1}=\delta\!\!\!\,\delta_1^\ast (({\sigma\!\!\!
\sigma_1^0})^{-1}\I_1\cdot \OO_{\Sigma_1^0}).
\end{eqnarray*}

The sheaf $\FFitt^0(\sigma\!\!\!\sigma_1^{0\ast} \EExt^1(W_1, \OO_{\Sigma_0}))=(\sigma_1^0)^{-1}\I_1 \cdot \OO_{\Sigma_1^0}$ is invertible, and by Lemma  \ref{mtlemma} we conclude that $\hd \sigma\!\!\!\sigma_1^{0\ast}\EExt^1(W_1, \OO_{\Sigma_0})=1.$

Till the end of the example with no reference to the form of the transformation $\sigma\!\!\!\sigma_i,$ $i=1,2,$ we use the decomposition  $\sigma\!\!\!\sigma _i=\sigma\!\!\!\sigma_1^0 \circ \delta\!\!\!\,\delta_i$. In the case when  $\sigma\!\!\!\sigma_1$ is the morphism of blowing up, the morphism  $\delta\!\!\!\,\delta_i$ is an identity morphism.

Hence the sheaf  $N_1=\ker({\sigma\!\!\!\sigma_1^0}^\ast \widehat E_2^{\vee} \to {\sigma\!\!\!\sigma_1^0}^\ast\EExt^1(W_1, \OO_{\Sigma_0}))$  is locally free. Then there is a morphism of locally free sheaves ${\sigma\!\!\!\sigma_1^0}^\ast \widehat E_1^\vee \to N_1$. The sheaf $H_1$ decomposes the morphism  $\widehat E_1 \to \widehat E_2$ into a composite of an epimorphism and a monomorphism. The sheaf $N_1$ decomposes the morphism  ${\sigma\!\!\!\sigma_1^0}^\ast H_1 \to {\sigma\!\!\!\sigma_1^0}^\ast \widehat E_2^\vee$ into a composite of an epimorphism and a monomorphism. The morphism ${\sigma\!\!\!\sigma_1^0}^\ast \widehat E_1^\vee \to {\sigma\!\!\!\sigma_1^0}^\ast H_1$ is an epimorphism. This implies that the composite  ${\sigma\!\!\!\sigma_1^0}^\ast\widehat E_1^\vee \to {\sigma\!\!\!\sigma_1^0}^\ast H_1 \to N_1$ is an epomorphism. Then the kernel of this composite is also locally free; denote $W_1':=\ker({\sigma\!\!\!\sigma_1^0}^\ast\widehat E_1^\vee \to {\sigma\!\!\!\sigma_1^0}^\ast H_1 \to N_1)^\vee.$ We come to the exact triple of locally free sheaves  $0\to {W_1'}^\vee \to {\sigma\!\!\!\sigma_1^0}^\ast \widehat E_1^\vee \to N_1 \to 0.$ Its dual is also exact, and we have a commutative diagram
\begin{equation}
\xymatrix{0\ar[r]& N_1^\vee \ar[r]& {\sigma\!\!\!\sigma_1^0}^\ast \widehat E_1 \ar[r]& W_1' \ar[r]&0\\
0\ar[r]& {\sigma\!\!\!\sigma_1^0}^\ast \widehat E_2 \ar[u] \ar[r]& {\sigma\!\!\!\sigma_1^0}^\ast \widehat E_1\ar@{=}[u] \ar[r]& {\sigma\!\!\!\sigma_1^0}^\ast W_1 \ar[u] \ar[r]&0}
\end{equation}
and an epimorphism  ${\sigma\!\!\!\sigma_1^0}^\ast W_1 \twoheadrightarrow W_1'$ to the pure quotient sheaf.

Applying the inverse image under the morphism  $\sigma\!\!\!\sigma_1$
to the exact triple $0\to W_1 \to \widehat E_0 \to \E \to 0$ we come to the 
$\OO_{\Sigma_1}$-triple $$0\to \delta\!\!\delta_1^{0\ast}W_1' \to
\sigma\!\!\!\sigma_1^\ast \widehat E_0 \to \sigma\!\!\!\sigma_1^\ast \E\to
0.$$ After similar manipulations of resolution with the morphism
$\sigma\!\!\!\sigma_2^0$ applied to this triple, we come to  $\OO_{\Sigma_2^0}$-sheaf
$N_2:=\ker(\sigma\!\!\!\sigma_2^{0\ast} {W_1'}^\vee \to
\sigma\!\!\!\sigma_2^{0\ast} \EExt^1( \sigma\!\!\!\sigma_1^\ast
\E,\OO_{\Sigma_1}))$.

Reasoning in analogous fashion as in the preceding step we conclude that the $0$-th Fitting ideal  
\begin{equation}\FFitt^0 (\sigma\!\!\!\sigma_2^{0\ast} \EExt^1(
\sigma\!\!\!\sigma_1^\ast
\E,\OO_{\Sigma_1}))=(\sigma\!\!\!\sigma_2^0)^{-1}\I_2 \cdot
\OO_{\Sigma_2^0} \label{fittex}\end{equation} is invertible. But the scheme  $\Sigma_2^0$, as well as the scheme  
$\Sigma_1$, is not obliged to have an irreducible reduction which is necessary to validate Lemma  \ref{mtlemma}. The sheaf of Fitting ideals  (\ref{fittex}) can be rewritten in the form  
\begin{eqnarray*}\FFitt^0 (\sigma\!\!\!\sigma_2^{0\ast} \EExt^1(
\sigma\!\!\!\sigma_1^\ast \E,\OO_{\Sigma_1}))=\FFitt^0
(\sigma\!\!\!\sigma_2^{0\ast} \delta\!\!\delta_1^{0\ast}\EExt^1(
\sigma\!\!\!\sigma_1^{0\ast}
\E,\OO_{\Sigma_1^0}))\\=(\sigma\!\!\!\sigma_2^0 \circ
\delta\!\!\delta_1^0)^{-1}\FFitt^0 (\EExt^1(
\sigma\!\!\!\sigma_1^{0\ast} \E,\OO_{\Sigma_1^0})) \cdot
\OO_{\Sigma_2^0}\\=(\delta\!\!\delta_1^{-1}\circ\sigma\!\!\!\sigma_2^{-1})^{-1}\FFitt^0
(\EExt^1(\sigma\!\!\!\sigma_1^{0\ast} \E,\OO_{\Sigma_1^0})) \cdot
\OO_{\Sigma_2^0}\end{eqnarray*} and it is invertible on the whole of the scheme  
$\Sigma_2^0$. 
Hence this sheaf is invertible on its component  $\Sigma_2^{-1}$, i.e. the sheaf
$$(\sigma\!\!\!\sigma_2^{-1})^{-1}\FFitt^0
(\EExt^1(\sigma\!\!\!\sigma_1^{0\ast} \E,\OO_{\Sigma_1^0})) \cdot
\OO_{\Sigma_2^{-1}}=\FFitt^0(\sigma\!\!\!\sigma_2^{-1\ast}
\EExt^1(\sigma\!\!\!\sigma_1^{0\ast} \E,\OO_{\Sigma_1^0}))$$ is also invertible.  Lemma  \ref{mtlemma} yields:  $$\hd
(\sigma\!\!\!\sigma_2^{-1\ast}
\EExt^1(\sigma\!\!\!\sigma_1^{0\ast} \E,\OO_{\Sigma_1^0}))=1.$$ 
This implies that it is possible to replace  $\delta\!\!\!\,\delta_1^0 \circ \sigma\!\!\!\sigma_2^0$ by $\sigma\!\!\!\sigma_2^{-1} \circ \delta\!\!\!\,\delta_1^{-1}$ with respect to the commutative square in the diagram  (\ref{doubletriangle}) and to do resolution procedure by the transformation under the morphism 
$\sigma\!\!\!\sigma_2^{-1}$ following after the transformation under 
$\sigma\!\!\!\sigma_1^0$. After all necessary manipulations we come to a sheaf $N'_2$:
$$N'_2:=\ker(\sigma\!\!\!\sigma_2^{-1\ast} {W_1'}^\vee \to
\sigma\!\!\!\sigma_2^{-1\ast} \EExt^1( \sigma\!\!\!\sigma_1^{0\ast}
\E,\OO_{\Sigma_1^0})).$$ Lemma \ref{mtlemma} is applicable in this circumstance and hence  $N'_2$ is locally free. We have a composite morphism  
$$\sigma\!\!\!\sigma_2^{-1\ast} \sigma\!\!\!\sigma_1^{0\ast}\widehat E_0^\vee \twoheadrightarrow
\sigma\!\!\!\sigma_2^{-1\ast}H'_2\twoheadrightarrow N'_2$$ (where $H'_2=\coker ((\sigma\!\!\!\sigma_1^{0\ast}\E)^\vee \to (\sigma\!\!\!\sigma_1^{0\ast}
\widehat E_0)^\vee)$).
 Its kernel is also a locally free $\OO_{\Sigma_2^{-1}}$-module. Denoting   $$\widetilde \E':=(\ker(\sigma\!\!\!\sigma_2^{-1\ast} \sigma\!\!\!\sigma_1^{0\ast}\widehat E_0^\vee \twoheadrightarrow
\sigma\!\!\!\sigma_2^{-1\ast}H'_2\twoheadrightarrow N'_2))^\vee,$$
we have  $\widetilde \E:=\delta\!\!\delta_2^{0\ast}
\delta\!\!\delta_1^{-1\ast} \widetilde \E'.$  This completes the example.
\end{example}


To generalize this recipe to an arbitrary homological dimension we act
inductively and do interchanging in the inductive step.

At the second step we pass to the next segment
$$
0\to W_{\ell-1} \to \widehat E_{\ell -2} \to W_{\ell -2 }\to 0
$$
and to its ``inverse image'' under $\sigma \!\!\! \sigma_1 ^0$:
$$0\to W'_{\ell -1} \to \sigma \!\!\! \sigma_1 ^{0\ast} \widehat E_{\ell
-2} \to \sigma \!\!\! \sigma_1 ^{0\ast} W_{\ell -2} \to 0.
$$
The ``inverse image'' under $\sigma \!\!\! \sigma_1 =\sigma \!\!\!
\sigma_1 ^0 \circ \delta \!\! \delta_1 ^0 $ is exact because the kernel $W'_{\ell-1}$
is locally free. In this and other
further segments the cokernel sheaves contain torsions. This does not cause any obstacle for the procedure of
resolution because, as we will see later, the resolution leads to
a factoring out of the torsion.

Next steps are similar to each other. They involve inverse images of
the consequent segments
\begin{eqnarray}
0\to W'_{\ell-2} \to \sigma \!\!\! \sigma_2^{0\ast}\sigma \!\!\!
\sigma_1^\ast \widehat E_{\ell-3} \to \sigma \!\!\! \sigma_2 ^{0\ast}\sigma
\!\!\! \sigma_1 ^\ast W_{\ell -3}
\to 0,\nonumber \\
.\quad .\quad .\quad .\quad .\quad .\quad .\quad .\quad .\quad
.\quad .\quad .\quad .\quad .\quad .\quad .\quad .\nonumber \\
0\to W'_{\ell-i} \to \sigma \!\!\! \sigma_i ^{0\ast}
\sigma\!\!\!\sigma_{i-1}^\ast\dots \sigma \!\!\! \sigma_1 ^\ast
\widehat E_{\ell-i-1} \to \sigma \!\!\! \sigma_i ^{0\ast}
\sigma\!\!\!\sigma_{i-1}^\ast \dots \sigma \!\!\! \sigma_1 ^\ast
W_{\ell -i-1}
\to 0,\label{v}\\
.\quad .\quad .\quad .\quad .\quad .\quad .\quad .\quad .\quad
.\quad .\quad .\quad .\quad .\quad .\quad .\quad .\nonumber \\
0\to W'_1 \to \sigma \!\!\! \sigma_{\ell-1} ^\ast\dots \sigma \!\!\!
\sigma_1 ^\ast \widehat E_0 \to \sigma \!\!\! \sigma_{\ell-1} ^\ast\dots
\sigma \!\!\! \sigma_1 ^\ast W_0 \to 0, \nonumber
\end{eqnarray}
where $W_0=\E$ and the kernel sheaf $W'_{\ell -i}$ in the next
triple is a locally free $\OO_{\Sigma_{i-1}}$-module which is produced
by the resolution of $\sigma \!\!\! \sigma_i ^\ast\dots \sigma
\!\!\! \sigma_1 ^\ast W_{\ell -i}$ in the previous triple.

Denote the composite $$\sigma\!\!\!\sigma_{[i]}:=\sigma
\!\!\! \sigma_1 \circ \dots \circ \sigma \!\!\! \sigma_i.$$
To describe the standard resolution of each triple in a uniform
fashion we need com\-pos\-i\-te morphisms
\begin{eqnarray*}
 \sigma\!\!\!\sigma_{[i-1]}\!\circ \!\sigma\!\!\!\sigma_i^0=\sigma\!\!\!\sigma_{[i-2]}\circ \sigma\!\!\!\sigma_{i-1}\circ \sigma\!\!\!\sigma_i^0=\sigma\!\!\!\sigma_{[i-2]}\circ \sigma\!\!\!\sigma_{i-1}^0\circ \delta\!\!\delta_{i-1}^0\circ \sigma\!\!\!\sigma_i^0=
\sigma\!\!\!\sigma_{[i-2]}\circ \sigma\!\!\!\sigma_{i-1}^0\circ  \sigma\!\!\!\sigma_i^{-1}\circ\delta\!\!\delta_{i-1}^{-1}=\\
\sigma\!\!\!\sigma_{[i-3]}\circ \sigma\!\!\!\sigma_{i-2}\circ \sigma\!\!\!\sigma_{i-1}^0\circ  \sigma\!\!\!\sigma_i^{-1}\circ\delta\!\!\delta_{i-1}^{-1}=
\sigma\!\!\!\sigma_{[i-3]}\circ \sigma\!\!\!\sigma_{i-2}^0\circ \delta\!\!\delta_{i-2}^0\circ \sigma\!\!\!\sigma_{i-1}^0\circ  \sigma\!\!\!\sigma_i^{-1}\circ\delta\!\!\delta_{i-1}^{-1}=\\
\sigma\!\!\!\sigma_{[i-3]}\circ \sigma\!\!\!\sigma_{i-2}^0\circ\sigma\!\!\!\sigma_{i-1}^{-1} \circ\delta\!\!\delta_{i-2}^{-1}\circ  \sigma\!\!\!\sigma_i^{-1}\circ\delta\!\!\delta_{i-1}^{-1}=\sigma\!\!\!\sigma_{[i-3]}\circ \sigma\!\!\!\sigma_{i-2}^0\circ\sigma\!\!\!\sigma_{i-1}^{-1} \circ   \sigma\!\!\!\sigma_i^{-2}\circ \delta\!\!\delta_{i-2}^{-2}\circ\delta\!\!\delta_{i-1}^{-1}=\\
\dots=\sigma\!\!\!\sigma_1^0\circ \sigma\!\!\!\sigma_2^{-1}\circ \dots \circ  \sigma\!\!\!\sigma_i^{-i+1}\circ \delta\!\!\delta_1^{-i+1}\circ \delta\!\!\delta_{2}^{-i+2} \circ \dots \circ \delta\!\!\delta_{i-1}^{-1}
=\sigma\!\!\!\sigma_{[i]}^0 \circ \delta\!\!\delta_{[i-1]}^{-1}
\end{eqnarray*}
and $$
\sigma\!\!\!\sigma_{[i]}=\sigma\!\!\!\sigma_{[i-1]}\circ \sigma\!\!\!\sigma_i=
\sigma\!\!\!\sigma_{[i-1]}\circ \sigma\!\!\!\sigma_i^0 \circ \delta\!\!\delta_i^0=
\sigma\!\!\!\sigma_{[i]}^0 \circ \delta\!\!\delta_{[i]}^0,
$$
where we have introduced the notations $$
\sigma\!\!\!\sigma_{[i]}^0:=\sigma\!\!\!\sigma_1^0 \circ \sigma\!\!\!\sigma_2^{-1} \circ \dots \circ
\sigma\!\!\!\sigma_i^{-i+1},\quad
\delta\!\!\delta_{[i]}^0:=\delta\!\!\delta_1^{-i+1}\circ \delta\!\!\delta_2^{-i+2}\circ  \dots \circ
\delta\!\!\delta_i^0,\quad \delta\!\!\delta_{[i-1]}^{-1}:=\delta\!\!\delta_1^{-i+1}\circ \delta\!\!\delta_{2}^{-i+2} \circ \dots \circ \delta\!\!\delta_{i-1}^{-1}.$$ 

We perform the standard resolution of the
triple
$$0\to
\delta\!\!\delta_{[i]}^{0\ast}W'_{\ell-i} \to \sigma \!\!\!
\sigma_{[i]} ^\ast \widehat E_{\ell-i-1} \to \sigma \!\!\! \sigma_{[i]} ^\ast
W_{\ell -i-1} \to 0
$$
for $i=1, \dots, \ell-1.$ Dualization leads to the exact
$\OO_{\Sigma_i}$-sequence
$$
0\to (\sigma \!\!\!\sigma_{[i]}^\ast W_{\ell-i-1})^{\vee} \to \sigma
\!\!\!\sigma_{[i]}^\ast \widehat E_{\ell-i-1}^{\vee} \to
\delta\!\!\delta_{[i]}^{0\ast}{W'}_{\ell-i}^{\vee} \to
\EExt^1(\sigma \!\!\! \sigma_{[i]} ^\ast W_{\ell -i-1},
\OO_{\Sigma_i}) \to 0.
$$
Denote \begin{eqnarray*}H_{i+1}:=\ker
(\delta\!\!\delta_{[i]}^{0\ast}{W'}_{\ell-i}^\vee \to
\EExt^1(\sigma\!\!\!\sigma_{[i]}^\ast W_{\ell-i-1},
\OO_{\Sigma_i}))\\=\coker((\sigma \!\!\!\sigma_{[i]}^\ast
W_{\ell-i-1})^{\vee} \to \sigma \!\!\!\sigma_{[i]}^\ast
\widehat E_{\ell-i-1}^{\vee}).\end{eqnarray*} 

Apply the inverse image $\sigma \!\!\! \sigma_{i+1}^{0\ast}$:
\begin{eqnarray}\xymatrix{
\sigma \!\!\! \sigma_{i+1}^{0\ast}(\sigma \!\!\!\sigma_{[i]}^\ast
W_{\ell-i-1})^{\vee}\ar[r]&\sigma \!\!\! \sigma_{i+1}^{0\ast}\sigma
\!\!\!\sigma_{[i]}^\ast \widehat E_{\ell-i-1}^{\vee} \ar[r]
&\sigma \!\!\! \sigma_{i+1}^{0\ast}H_{i+1} \ar[r]&0, }  \nonumber\\
\label{c0}\xymatrix{ \sigma \!\!\! \sigma_{
i+1}^{0\ast}H_{i+1}\ar[r]&\sigma \!\!\! \sigma_{
i+1}^{0\ast}\delta\!\!\delta_{[i]}^{0\ast}{W'}_{\ell-i}^{\vee}\ar[r]
&\sigma \!\!\! \sigma_{ i+1}^{0\ast}\EExt^1(\sigma \!\!\!
\sigma_{[i]} ^\ast W_{\ell -i-1}, \OO_{\Sigma_i})\ar[r]&0.}
\end{eqnarray}

In  (\ref{c0}) denote $N_{i+1}:=\ker(\sigma \!\!\! \sigma_{
i+1}^{0\ast}\delta\!\!\delta_{[i]}^{0\ast}{W'}_{\ell-i}^{\vee}\to
\sigma \!\!\! \sigma_{ i+1}^{0\ast}\EExt^1(\sigma \!\!\!
\sigma_{[i]} ^\ast W_{\ell -i-1}, \OO_{\Sigma_i}))$. The sheaf \linebreak
$\FFitt^0(\sigma \!\!\! \sigma_{i+1}^{0\ast}\EExt^1(\sigma \!\!\!
\sigma_{[i ]} ^\ast W_{\ell -i-1}, \OO_{\Sigma_i}))$ is invertible
by the functorial property of $\FFitt$:
\begin{eqnarray}\label{fitti}\FFitt ^0(\sigma \!\!\! \sigma_{
i+1}^{0\ast}\EExt^1(\sigma \!\!\! \sigma_{[i]} ^\ast W_{\ell -i-1},
\OO_{\Sigma_{i+1}}))=(\sigma\!\!\!\sigma^0_{i+1})^{-1} \I_{i+1}\cdot
\OO_{ \Sigma_{i+1}}.
\end{eqnarray}

Now we need the following easy
\begin{lemma} There is an isomorphism \begin{equation}
 \EExt^1(\sigma\!\!\!\sigma_{[i]}^\ast W_{\ell-i-1}, \OO_{\Sigma_i})=
 \delta\!\!\delta_{[i]}^{0\ast}\EExt^1(\sigma\!\!\!\sigma_{[i]}^{0\ast}W_{\ell-i-1}, \OO_{ \Sigma_{i}^{-i+1}}).
\end{equation}
\end{lemma}
\begin{proof}
Start with the exact triple 
\begin{equation}\label{trii}
0\to W'_{\ell-i} \to \sigma\!\!\!\sigma_{[i]}^{0\ast} \widehat E_{\ell-i-1} \to \sigma\!\!\!\sigma_{[i]}^{0\ast} W_{\ell-i-1}\to 0 
\end{equation}
and with its inverse $\delta\!\!\delta_{[i]}^{0}$-image 
\begin{equation}\label{deltatrii}
0\to \delta\!\!\delta_{[i]}^{0\ast}W'_{\ell-i} \to \delta\!\!\delta_{[i]}^{0\ast}\sigma\!\!\!\sigma_{[i]}^{0\ast} \widehat E_{\ell-i-1} \to \delta\!\!\delta_{[i]}^{0\ast}\sigma\!\!\!\sigma_{[i]}^{0\ast} W_{\ell-i-1}\to 0 .    
\end{equation}
The triple (\ref{deltatrii}) is also exact because the sheaf  $W'_{\ell-i}$ is locally free (locally free sheaves $W'_{\ell-i}$ are obtained inductively on $i$).
Dualization of both the triples (\ref{trii},\ref{deltatrii}) and taking an inverse image of the dual of (\ref{trii}) lead to the diagram with exact rows and vertical isomorphisms 
$$\xymatrix{\delta\!\!\delta_{[i]}^{0\ast} (\sigma\!\!\!\sigma_{[i]}^{0\ast} \widehat E_{\ell-i-1})^\vee \ar@{=}[d] \ar[r]& \delta\!\!\delta_{[i]}^{0\ast}(W'_{\ell-i})^\vee\ar@{=}[d] \ar[r]&
\delta\!\!\delta_{[i]}^{0\ast}\EExt^1(\sigma\!\!\!\sigma_{[i]}^{0\ast}W_{\ell-i-1},
 \OO_{\Sigma_{i}^{-i+1}}) \ar@{=}[d]\ar[r]&0\\
(\delta\!\!\delta_{[i]}^{0\ast}\sigma\!\!\!\sigma_{[i]}^{0\ast} \widehat E_{\ell-i-1})^\vee \ar[r]&
(\delta\!\!\delta_{[i]}^{0\ast}W'_{\ell-i})^\vee \ar[r] &\EExt^1(\sigma\!\!\!\sigma_{[i]}^\ast W_{\ell-i-1}, \OO_{\Sigma_i}) \ar[r]&0}
$$
The right hand side vertical isomorphism completes the proof.
\end{proof}

Now (\ref{fitti}) can be rewritten as
\begin{eqnarray*}\FFitt^0 (\sigma\!\!\!\sigma_{i+1}^{0\ast} \EExt^1(
\sigma \!\!\! \sigma_{[i]} ^\ast W_{\ell -i-1},
\OO_{\Sigma_{i+1}})))=\FFitt^0 (\sigma\!\!\!\sigma_{i+1}^{0\ast}
\delta\!\!\delta_{[i]}^{0\ast}\EExt^1(
\sigma\!\!\!\sigma_{[i]}^{0\ast}
W_{\ell-i-1},\OO_{\Sigma_i^{-i+1}}))\\=(\delta\!\!\delta_{[i]}^0
\circ \sigma\!\!\!\sigma_{i+1}^0)^{-1}\FFitt^0 (\EExt^1(
\sigma\!\!\!\sigma_{[i]}^{0\ast}
W_{\ell-i-1},\OO_{\Sigma_i^{-i+1}})) \cdot
\OO_{\Sigma_{i+1}^0}\\
=(\sigma\!\!\!\sigma_{i+1}^{-i}\circ
\delta\!\!\delta_{[i]}^{-1})^{-1}\FFitt^0
(\EExt^1(\sigma\!\!\!\sigma_{[i]}^{0\ast}
W_{\ell-i-1},\OO_{\Sigma_i^{-i+1}})) \cdot
\OO_{\Sigma_{i+1}^0}\end{eqnarray*} (where we use the notation
$\delta\!\!\delta_{[i]}^{-1}:=\delta\!\!\delta_1^{-i}\circ \dots
\circ \delta\!\!\delta_i^{-1}$). This sheaf is invertible on the scheme $\Sigma_{i+1}^0$ in
whole. Hence it is invertible on the
component $\Sigma_{i+1}^{-i}$ of the scheme $\Sigma_{i+1}^0$, i.e. the sheaf
$$(\sigma\!\!\!\sigma_{i+1}^{-i})^{-1}\FFitt^0
(\EExt^1(\sigma\!\!\!\sigma_{[i]}^{0\ast}
W_{\ell-i-1},\OO_{\Sigma_i^{-i+1}})) \cdot \OO_{\Sigma_{i+1}^{-i}}$$
is also invertible.  The scheme $\Sigma_{i+1}^{-i}$ has an irreducible reduction and application of Lemma~\ref{mtlemma} yields:  
\begin{equation}\label{hd}\hd
(\sigma\!\!\!\sigma_{i+1}^{-i})^\ast
\EExt^1(\sigma\!\!\!\sigma_{[i]}^{0\ast}
W_{\ell-i-1},\OO_{\Sigma_i^{-i+1}})=1.
\end{equation} This means that it is
possible to interchange the morphisms $$\delta\!\!\delta_{[i]}^0 \circ
\sigma\!\!\!\sigma_{i+1}^0=\sigma\!\!\!\sigma_{i+1}^{-i}\circ
\delta\!\!\delta_{[i]}^{-1}$$ and to do the resolution by
$(\sigma\!\!\!\sigma_{i+1}^{-i})^\ast$ following
$(\sigma\!\!\!\sigma_{[i]}^0)^\ast$.

We perform all the necessary steps of the resolution beginning with the triple
$$
0\to W'_{\ell-i} \to \sigma\!\!\!\sigma_{[i]}^{0\ast}\widehat E_{\ell-i-1} \to \sigma\!\!\!\sigma_{[i]}^{0\ast}W_{\ell-i-1}\to 0
$$
(generally, over all $i$ inductively along the bottom of the diagram (\ref{bigdiag})).
Denoting 
\begin{eqnarray*}
H'_{i+1}:=\ker ({W'_{\ell-i}}^\vee \to \EExt^1(\sigma\!\!\!\sigma^{0\ast}_{[i]}W_{\ell-i-1}, \OO_{\Sigma_i^{-i+1}}))\\
=\coker ((\sigma\!\!\!\sigma_{[i]}^{0\ast}W_{\ell-i-1})^\vee \to (\sigma\!\!\!\sigma_{[i]}^{0\ast}\widehat E_{\ell-i-1})^\vee)
\end{eqnarray*}
and applying $\sigma\!\!\!\sigma_{i+1}^{-i\ast}$ in view of (\ref{hd}),
we come to the locally free sheaf
$N'_{i+1}$: 
\begin{equation}\label{Nl-i} N'_{i+1}:=\ker(\sigma\!\!\!\sigma_{i+1}^{-i\ast} {W'_{\ell-i}}^\vee \to
\sigma\!\!\!\sigma_{i+1}^{0\ast}
\EExt^1(\sigma\!\!\!\sigma_{[i]}^{0\ast}
W_{\ell-i-1},\OO_{\Sigma_i^{-i+1}})).\end{equation} Now we have a
composite morphism
$$(\sigma\!\!\!\sigma_{[i+1]}^{0\ast}) \widehat E_{\ell-i-1}^\vee \twoheadrightarrow
\sigma\!\!\!\sigma_{i+1}^{-i\ast}H'_{i+1}\twoheadrightarrow N'_{i+1}$$ of locally free $\OO_{\Sigma_{i+1}^{-i}}$-sheaves. 
The kernel of this composite morphism is also locally free. We denote it as ${W'_{\ell-i-1}}^\vee$ 
and from (\ref{Nl-i}) come to the monomorphism $\sigma\!\!\!\sigma_{i+1}^{-i\ast}W'_{\ell-i} \hookrightarrow {N'_{i+1}}^\vee.$
This leads to the commutative diagram with exact rows
$$\xymatrix{0 \ar[r]& \sigma\!\!\!\sigma_{i+1}^{-i\ast}W'_{\ell-i} \ar@{^(->}[d] \ar[r]&
\sigma\!\!\!\sigma_{[i+1]}^{0\ast} \widehat E_{\ell-i-1} \ar@{=}[d] \ar[r]& \sigma\!\!\!\sigma_{[i+1]}^{0\ast}W_{\ell-i-1} \ar@{->>}[d] \ar[r]&0\\
0 \ar[r]& {N'_{i+1}}^\vee \ar[r]& \sigma\!\!\!\sigma_{[i+1]}^{0\ast} \widehat E_{\ell-i-1} \ar[r]& W'_{\ell-i-1} \ar[r]&0}
$$
and $W'_{\ell-i-1}=(\sigma\!\!\!\sigma_{[i+1]}^{0\ast}W_{\ell-i-1})/\tors $. This completes the inductive step of the resolution procedure.

After all the steps one comes to a locally free $\OO_{\Sigma_\ell^{-\ell+1}}$-sheaf \begin{equation} \label{sh1}
   \widehat \E:= \sigma\!\!\!\sigma_{[\ell]}^{0\ast} \E /\tors
\end{equation}
and an $\OO_{\Sigma_\ell}$-sheaf 
\begin{equation} \label{sh2}
\widetilde \E:=\delta\!\!\delta_{[\ell]}^{0\ast}\widehat \E.    
\end{equation} 
We will use the notation $\widetilde \Sigma:=\Sigma_\ell$ which is analogous to one used for a family of admissible schemes in the previous papers.

\begin{remark} Since  $\widetilde \E$ is locally free as 
 $\OO_{\widetilde
\Sigma}$-module and  $\OO_{\widetilde \Sigma}$ is  a $\OO_{\widetilde T}$-flat $\OO_{\widetilde \Sigma}$-module, then
$\widetilde \E$ is flat over $\widetilde T$.
\end{remark}

\begin{remark}
The procedure described involves a choice of locally free resolution of the sheaf under the resolution. But the sequence of morphisms does not depend on this choice due to the following theorem from commutative algebra.
\end{remark}
\begin{theorem}[\cite{Eisen}, Theorem 20.2] Let $R$ be a local ring and $M$ be a finitely generated $R$-module. If $F$ is a minimal free resolution of $M$ then any free resolution of $M$ is isomorphic to the direct sum of $F$ and a trivial complex. In particular, there is  only one (up to isomorphism) minimal free resolution of $M$.
\end{theorem}
From this theorem we conclude  inductively that the  sheaves $$\EExt^1(W_{\ell-1},\OO_{\Sigma_1}),\;\;\EExt^1(\sigma\!\!\!\sigma_\ell ^\ast W_{\ell-2},\OO_{\Sigma_2}), \; \dots,\;\EExt^1(\sigma\!\!\!\sigma_2 ^\ast\dots \sigma\!\!\!\sigma_\ell ^\ast W_0, \OO_{\Sigma_{\ell}})$$ are independent of the choice of a locally free resolution.

\section{The structure of an admissible scheme and the fiberwise resolution} \label{s4}
In this section we discuss  the structure of fibers of the morphism  $\widetilde \pi \colon \widetilde \Sigma \to \widetilde T$. 
The resolution procedure leads to the composite of morphisms of $\widetilde T$-schemes $\pi_i\colon \Sigma_i \to T$, $i=0,\dots,\ell$: 
$$\xymatrix{\widetilde \Sigma \ar@{=}[r]&\!\Sigma_\ell \ar[rd]_{\delta\!\!\delta_\ell^0} \ar[rr]^{\sigma\!\!\!\sigma_\ell}&&\Sigma_{\ell-1} \ar[rd]_{\delta\!\!\delta_{\ell-1}^0}\ar[rr]^{\sigma\!\!\!\sigma_{\ell-1}}&&\dots \ar[rr]^{\sigma\!\!\!\sigma_2}\ar[rd]&&\Sigma_1 \ar[dr]_{\delta\!\!\delta_1^0}\ar[rr]^{\sigma\!\!\!\sigma_1}&&\Sigma_0 \ar@{=}[r]&\Sigma\\
&&\Sigma_{\ell}^0\ar[ur]_{\sigma\!\!\!\sigma_\ell^0}&&\Sigma_{\ell-1}^0\ar[ur]&&\Sigma_2^0 \ar[ur]_{\sigma\!\!\!\sigma_2^0}&& \Sigma_1^0 \ar[ur]_{\sigma\!\!\!\sigma_1^0}}
$$
When restricted to the fibers over any fixed closed point $t\in T$ (or, equivalently, one can set $\widetilde T$ to be a reduced point) this chain gives rise to the chain of morphisms among fibers~$S_i=\pi_i^{-1}(t)$:
$$\xymatrix{\widetilde S \ar@{=}[r]&S_\ell \ar[rd]_{\delta_\ell^0} \ar[rr]^{\sigma_\ell}&&S_{\ell-1} \ar[rd]_{\delta_{\ell-1}^0}\ar[rr]^{\sigma_{\ell-1}}&&\dots \ar[rr]^{\sigma_2}\ar[rd]&&S_1 \ar[dr]_{\delta_1^0}\ar[rr]^{\sigma_1}&&S_0 \ar@{=}[r]&S\\
&&S_{\ell}^0\ar[ur]_{\sigma_\ell^0}&&S_{\ell-1}^0\ar[ur]&&S_2^0 \ar[ur]_{\sigma_2^0}&& S_1^0 \ar[ur]_{\sigma_1^0}}
$$
Each of the morphisms $\sigma_i^0$ is a blowing up morphism and each of the morphisms $\delta_i^0$ contracts the corresponding additional component of the scheme $S_i$ onto the exceptional divisor of the morphism $\sigma_i^0$. Moving against arrows we can say that the morphisms $\sigma_i^0$ blow up and the morphisms $\delta_i^0$ grow up the corresponding additional components. Since these two types of morphisms alternate, next blowup is applied to the scheme which consists of several connected components.

\begin{convention}
We use the notations for  single fibers which are completely parallel to the ones for families, preserving the indexing rules as they were used in (\ref{bigdiag}). The fiberwise version for (\ref{bigdiag}) can be obtained when one replaces the symbol $\Sigma$ by the symbol $S$ with all indices preserved. Also, when passing to the morphisms for single fibers, double letters $\sigma\!\!\!\sigma $, $\delta\!\!\delta$ are replaced by respective usual ``single'' letters $\sigma$, $\delta$ with all indices preserved for morphisms.  Each $\delta$ projects its source scheme to its component.
 \end{convention}

We start with the initial nonsingular variety $S_0=S$. When passing to $S_1^0$ and after that to $S_1$, we arrive to the scheme $S_1$ consisting of a principal component $S_1^0$ and an addit\-ional component(s) $S_1^{add}$. The principal component $S_1^0=(\sigma_1^0)^{-1}S$ is an algebraic variety which is obtained by blowing up of $S$. The addit\-ional component $S_1^{add}$ can carry a nonreduced scheme structure and it also can have a reducible reduction. For the sake of convenience we refer to the subscheme  $S_1^{add}$, as well as the subscheme  $S_i^{add}$ which is obtained in analogous way at $i$-th step of the transformation, as to an additional component. In the previous papers this closed subscheme appeared as a union of addit\-ional components of the admissible scheme.

Passing to $S_2^0$ leads to a transformation of both $S_1^0$ and $S_1^{add}$. We come to an algebraic variety $S_2^0=(\sigma_2^{-1})^{-1}S_1^0=(\sigma_2^{-1})^{-1}(\sigma_1^0)^{-1}S$ obtained by blowing up of the principal component $S_1^0$ of $S_1$ and to the scheme $(\sigma_2^0)^{-1}S_1^{add}$. Passing to $S_2$ against $\delta_2^0$ grows up an additional component $S_2^{add}$, and we can write $S_2=S_2^0 \cup (\sigma_2^0)^{-1}S_1^{add} \cup S_2^{add}$. 

Analogously, on $\ell$-th step we have the following scheme 
$$
\widetilde S:=S_\ell=(\sigma_1^0\circ \dots \circ \sigma_\ell^{-\ell+1})^{-1}S \cup 
(\sigma_2^0\circ \dots \circ \sigma_\ell^{-\ell+2})^{-1}S_1^{add}\cup \dots \cup (\sigma_\ell^0)^{-1}S_{\ell-1}^{add}\cup S_\ell^{add}.$$

Depending on the structure of the initial  $\OO_S$-sheaf $E$, several morphisms $\sigma_i$ can turn to be  identities. Hence, the actual length of the chain 
$S_1, \dots, S_\ell$ can vary from 0 (for the case when $E$ is locally free) to the maximal value equal to $\ell=\hd \E.$

To measure numerical invariants of the objects we obtained in the standard resolution we need to fix an appropriate ample invertible sheaf $\widetilde L$ on each $\widetilde S$. The sheaf playing an analogous role was called a {\it distinguished polarization} in previous papers, where the procedure of standard resolution in the $\hd$-one case was developed. These invertible sheaves provide, in particular, fiberwise uniform Hilbert polynomials in flat families of admissible schemes. Strictly speaking, if $\widetilde \L$ is an invertible $\OO_{\widetilde \Sigma}$-sheaf which is very ample relatively to the base $\widetilde T$, then for any closed point $\widetilde t\in \widetilde T$ and for any integer $n\gg 0$ the Hilbert polynomial of the fiber $\widetilde \pi^{-1}(\widetilde t)$ compute as $\chi (\widetilde \L^n|_{\widetilde \pi^{-1}(\widetilde t)})$ is independent of the choice of $\widetilde t\in \widetilde T$.

Let $\Sigma$ carries an invertible sheaf $\L$ which is very ample relatively to the base $\widetilde T$ of the family $\Sigma$. In $\hd$-one case there is a one-step standard resolution by $\widetilde T$-morphism $\sigma\!\!\!\sigma\colon \widetilde \Sigma \to \Sigma$. This resolution procedure is associated to the sheaf of ideals $\I \subset \OO_\Sigma$. In this case it was shown that a distinguished polarization can be chosen as $\widetilde \L=\sigma\!\!\!\sigma^\ast \L^m \otimes \sigma\!\!\!\sigma^{-1}\I \cdot \OO_{\widetilde \Sigma}$, when $m$ is sufficiently big to provide ampleness of $\widetilde \L$ relatively to $\widetilde T$.

In the case of the higher homological dimension $\ell$ this step of constructing of a family of polarizations on $\widetilde T$-scheme $\widetilde \Sigma$ is iterated $\ell$ times until one comes to 
\begin{equation} \label{polarizfam}
    \widetilde \L:=\L_\ell=[\sigma\!\!\!\sigma_\ell^\ast \dots  [\sigma\!\!\!\sigma_2^\ast[\sigma\!\!\!\sigma_1^\ast \L^{m_1} \otimes (\sigma\!\!\!\sigma_1)^{-1} \I_1 \cdot \OO_{\Sigma_1}]^{m_2} \otimes (\sigma\!\!\!\sigma_2)^{-1}\I_2 \cdot \OO_{\Sigma_2}]^{m_3}\dots ]^{m_\ell}\otimes (\sigma\!\!\!\sigma_\ell)^{-1} \I_\ell \cdot \OO_{\Sigma_\ell}
\end{equation}
The corresponding $\widetilde T$-scheme $\Sigma_i$ arising on the $i$-th step of the standard resolution has the fiberwise constant Hilbert polynomial with respect to 
\begin{equation*} 
    \L_i:=[\sigma\!\!\!\sigma_i^\ast \dots  [\sigma\!\!\!\sigma_2^\ast[\sigma\!\!\!\sigma_1^\ast \L^{m_1} \otimes (\sigma\!\!\!\sigma_1)^{-1} \I_1 \cdot \OO_{\Sigma_1}]^{m_2} \otimes (\sigma\!\!\!\sigma_2)^{-1}\I_2 \cdot \OO_{\Sigma_2}]^{m_3}\dots ]^{m_i}\otimes (\sigma\!\!\!\sigma_i)^{-1} \I_i \cdot \OO_{\Sigma_\ell}.
\end{equation*}

A distinguished polarization on a single admissible scheme $\widetilde S$ has a view 
\begin{equation}\label{polarizfib}
\widetilde L:=L_\ell=[\sigma_\ell^\ast \dots  [\sigma_2^\ast[\sigma_1^\ast L^{m_1} \otimes (\sigma_1)^{-1} I_1 \cdot \OO_{S_1}]^{m_2} \otimes (\sigma_2)^{-1}I_2 \cdot \OO_{S_2}]^{m_3}\dots ]^{m_\ell}\otimes (\sigma_\ell)^{-1} I_\ell \cdot \OO_{S_\ell}
\end{equation}
The distinguished polarization $\widetilde L$ is assumed to be fixed for each admissible scheme $\widetilde S$. If $\widetilde S=S$, then $\widetilde L=L$. 

\begin{convention}\label{poldeg}
Now we redenote $\L^{m_1\cdot m_2 \cdot \dots \cdot m_\ell}$ as $\L$  and $L^{m_1\cdot m_2 \cdot \dots \cdot m_\ell}$ as $L$ and from now we work with these new polarization. Also we come to the following shorthand notations for (\ref{polarizfam}) and (\ref{polarizfib}) respectively: 
\begin{eqnarray}
\widetilde \L=\sigma\!\!\!\sigma^\ast \L \otimes \E \mathrm{xc};\\ \E \mathrm{xc}:=[\dots [[ (\sigma\!\!\!\sigma_1)^{-1} \I_1 \cdot \OO_{\Sigma_1}]^{m_2} \otimes (\sigma\!\!\!\sigma_2)^{-1}\I_2 \cdot \OO_{\Sigma_2}]^{m_3}\dots ]^{m_\ell}\otimes (\sigma\!\!\!\sigma_\ell)^{-1} \I_\ell \cdot \OO_{\Sigma_\ell};
\end{eqnarray}
\begin{eqnarray}
\widetilde L=\sigma^\ast L \otimes \mathrm{Exc}_{S_\ell};\label{polshort}\\\mathrm{Exc}_{S_\ell}:=[\dots [[ (\sigma_1)^{-1} I_1 \cdot \OO_{S_1}]^{m_2} \otimes (\sigma_2)^{-1}I_2 \cdot \OO_{S_2}]^{m_3}\dots ]^{m_\ell}\otimes (\sigma_\ell)^{-1} I_\ell \cdot \OO_{S_\ell}. \label{polshortexc}
\end{eqnarray}\end{convention}
Now we come to the equality for the fibrewise Hilbert polynomials $$
\chi(\widetilde L^n)=\chi(L^n)=rp(n), \quad n\gg 0.
$$

The collection of degrees $m_i,\, i=1, \dots,\ell$, defining the change of notation in the Convention  \ref{poldeg}, is to be chosen common for the resolutions of all families of coherent sheaves with given Hilbert polynomial $rp(n)$.  Now confirm that it is possible.

\begin{definition}[\cite{HL}, Definition 1.7.5]  A set of isomorphism classes of coherent sheaves on a projective $k$-scheme~$S$ is {\it bounded}
if there is a $k$-scheme $B$ of finite type and a coherent $\OO_{B\times S}$-sheaf $\F$ such that the given
set is contained in the set $\{\F|_{b\times S}|b\in B \;\,\mbox{\rm  is a closed point}\}$.
\end{definition}
\begin{proposition} Let some set $Q$ of coherent $\OO_S$-sheaves $E$ which are deformation equivalent to locally free sheaves and  with Hilbert polynomials (computed with respect to the polarization $L$) equal to $P(n)$ be bounded. Then
polarizations $\widetilde L$ for all possible schemes $\widetilde S$ obtained by standard resolutions of all the sheaves $E\in Q$ can be chosen in such a way that all the admissible schemes $(\widetilde S, \widetilde L)$ have Hilbert polynomials also equal to~$P(n)$.
\end{proposition}
\begin{proof} Since the definition of bounded  set of sheaves involves only closed points of the scheme $B$, we can assume this scheme to be reduced. We are interested in $\OO_S$-sheaves with equal Hilbert polynomials. Hence, after excluding, if necessary, some closed subscheme~$B_0$ from $B$ we can suppose the sheaf $\F$ to be flat over $B\setminus B_0$. Reducedness of the scheme~$B\setminus B_0$ is needed to conclude flatness from fibrewise uniform  Hilbert polynomials. The open subscheme $B\setminus B_0$ contains all the closed points $b\in B$ such that $\F|_{b\times S} \in Q$.  Now one can set $T:=B\setminus B_0$ and perform the standard resolution of the family $\E=\F|_{T\times S}$. This resolution leads to the family of schemes $\widetilde \pi\colon \widetilde \Sigma \to \widetilde T$ with the family of distinguished polarizations~$\widetilde \L$ such that the fibrewise Hilbert polynomial $\chi(\widetilde \L^n|_{\widetilde \pi^{-1}(\widetilde b)})$ is uniform over $\widetilde b\in \widetilde T$ and, by deformation equivalence to pairs with $\widetilde S \cong S,$ it equals~$P(n)$.  
\end{proof}

\section{Further remarks on the additional components}\label{s5}

Each admissible scheme $\widetilde S=S_\ell$ is built iteratively by alternating transformations of two types: these are blowing ups $\sigma_i^0$ followed by growing ups of additional components $\delta_i=\delta_i^0$, $i=1,\dots, \ell$. Additional component $S_i^{add}=\Proj \bigoplus_{s\ge 0}(I_i[t]+(t))^s/(I_i[t]+(t))^{s+1}$ grown by $\delta_i$ intersects the union $S_i^0$ of all the components of the  previous ($(i-1)^{\mathrm{st}}$) level along an exceptional divisor $D_i=\Proj \bigoplus_{s\ge 0}I_i^s/I_i^{s+1}$ of the blowing up morphism $\sigma_i^0$.

For now we work with a commutative  ring $A$ and a proper ideal $I\subset A$; denote $\widehat A:= \bigoplus_{s\ge 0} (I^s/I^{s+1})$. Also set $M=\bigoplus_{s\ge 0}(I[t]+(t))^s/(I[t]+(t))^{s+1}=\bigoplus_{s\ge 0} M_s$ for $M_s=(I[t]+(t))^s/(I[t]+(t))^{s+1}$.
\begin{proposition}
$M=\bigoplus_{s\ge 0}(I[t]+(t))^s/(I[t]+(t))^{s+1}$ is flat as an $\widehat A$-module. 
\end{proposition}
\begin{proof}
Now write down explicit forms for every direct summand of $M$ keeping in mind that $t$ is a transcendental element of graded degree 1 over $A$. There are $A$-isomorphisms
\begin{eqnarray}
&&M_0=A[t]/(I[t]+(t))\cong A/I;\nonumber\\ &&M_1=(I[t]+(t))/(I[t]+(t))^2\cong I/I^2 \oplus (A/I)t;\nonumber\\
&&M_2=(I[t]+(t))^2/(I[t]+(t))^3\cong I^2/I^3 \oplus (I/I^2)t \oplus (A/I)t^2;\nonumber\\
&&.\;.\;.\;.\;.\;.\;.\;.\;.\;.\;.\;.\;.\;.\;.\;.\;.\;.\;.\;.\;.\;.\;.\;.\;.\;.\;.\;.\;.\;.\;.\;.\;.\;.\;.\;.\;.\;.\;.\;.\;.\;.\;.\;.\;.\;.\; \label{Mi}\\
&&M_s=(I[t]+(t))^s/(I[t]+(t))^{s+1}\cong I^s/I^{s+1}\oplus (I^{s-1}/I^s)t \oplus \dots \oplus (A/I)t^s;\nonumber\\
&&.\;.\;.\;.\;.\;.\;.\;.\;.\;.\;.\;.\;.\;.\;.\;.\;.\;.\;.\;.\;.\;.\;.\;.\;.\;.\;.\;.\;.\;.\;.\;.\;.\;.\;.\;.\;.\;.\;.\;.\;.\;.\;.\;.\;.\;.\;\nonumber
\end{eqnarray}
Here the expressions $(I^s/I^{s+1})t^q$ mean  additive subgroups of the view $\{\overline f t^q| \overline f \in I^s/I^{s+1}\}$.  
\begin{example}
In $M_1$ one has $i_0+i_1 t+\dots +i_s t^s+t(f_0+f_1 t+\dots + f_q t^q)\!\!\mod (I^2[t]+t I[t]+(t^2))\cong 
i_0 \!\!\mod I^2 + tf_0\!\!\mod I$ for $i_r\in I$, $r=0, \dots, s$; $f_u\in A$, $u=0, \dots ,q$. 
\end{example}
Taking in (\ref{Mi}) a direct sum over all $s\ge 0$ one comes to the expression
\begin{eqnarray*}
M&=&\bigoplus_{s\ge 0}(I[t]+(t))^s/(I[t]+(t))^{s+1} \nonumber\\
&\cong &(\bigoplus_{s\ge 0}I^s/I^{s+1})\oplus(\bigoplus_{s\ge 0}I^s/I^{s+1})t \oplus (\bigoplus_{s\ge 0}I^s/I^{s+1})t^2 \oplus (\bigoplus_{s\ge 0}I^s/I^{s+1})t^q \oplus \dots  \nonumber\\
&\cong & \widehat A \oplus \widehat A t \oplus \widehat A t^2 \oplus \dots \oplus \widehat A t^q \oplus \dots \label{flat}
\end{eqnarray*}
This expression shows that $M$ is an $\widehat A$-flat module.
\end{proof}

Set in (\ref{sh1}, \ref{sh2}) the base scheme $T$ to be a single reduced point and hence $\Sigma=S$, $\Sigma_i=S_i$ is an admissible scheme with the natural decomposition $S_i=S_i^0 \cup S_i^{add}$. The  morphism $\delta_i$ of additional component(s) factors through an exceptional divisor $D_i$ of the blowing up morphism $\sigma_i\colon S_i^0\to S_{i-1}$
$$\delta_i\colon S_i^{add} \to D_i \hookrightarrow S_i^0.
$$
Now we know that $\delta_i$ is flat over its scheme-theoretic image $D_i$. Then we can rewrite (\ref{sh1},\ref{sh2}) as $\widetilde E=\sigma^\ast E/\tors $ where the symbol $\tors$ denotes the torsion in the modified sense. It is defined inductively as described below (for the convenience we keep the notations $\widetilde S^0:=(\sigma_1^0\circ \dots \circ \sigma_\ell^{-\ell+1})^{-1}S$ for the principal component and $\widetilde S^{i}=(\sigma_{i+1}^0\circ \dots \circ \sigma_\ell^{-\ell+2})^{-1}S_i^{add}$ for all $i=1, \dots, \ell-1$).

Let $U$ be a Zariski-open subset in one of the 
components $\widetilde S^0$ or $\widetilde S_{j}, j> 0$, and
$\sigma^{\ast}E|_{\widetilde S^{j}}(U)$ a correspond\-ing group of
sections. This group is an $\OO_{\widetilde S^{j}}(U)$-module.
Let $\tors_j(U)$ be the submodule in $\sigma^{\ast}E|_{\widetilde S^j}(U)$ which consists of the sections $s\in \sigma^{\ast}E|_{\widetilde S^j}(U)$ such that $s$ is annihilated by some prime ideal of positive codimension in ${\mathcal O}_{\widetilde S^j}(U)$. The cor\-respondence
 $U \mapsto \tors_j(U)$ defines a subsheaf $\tors_j
\subset \sigma^{\ast}E|_{\widetilde S^{j}}.$ Note that associated
primes of positive codimensions which annihilate the sections $s\in
\sigma^{\ast}E|_{\widetilde S^{j}}(U)$ correspond to subschemes
supported in  $\widetilde S^{add}=\bigcup_{j>0}\widetilde S^{j}.$
Since the scheme  $\widetilde S=\widetilde S^0\cup \widetilde S^{add}$ is connected by the construction done in \cite{Tim3}, the subsheaves
$\tors_j, j\ge 0,$ allow us to construct a subsheaf $\tors \subset
\sigma^{\ast}E$. The latter subsheaf is defined as follows. A
section  $s\in \sigma^{\ast}E|_{\widetilde S^{j}}(U)$ satisfies
the condition $s\in \tors|_{\widetilde S^{j}}(U)$ if and only if\\
$(i)$ there exists a nonzero section  $y\in \OO_{\widetilde S^{j}}(U)$ such
that $ys=0$,\\
$(ii)$ for $j>1$ at least one of the following two requirements is
satisfied: either $y\in {\mathfrak p}$, where $\mathfrak p$ is a prime
ideal of positive codimension, or there exists a Zariski-open subset
 $V\subset \widetilde S$ and a section $s' \in \sigma^{\ast}E (V)$ such that
 $V\supset U$,
$s'|_U=s$ and for  $s'|_{V\cap \widetilde S^{j-1}}$ also
conditions analogous to $(i), (ii)$ hold. For $j=1$ conditions $(ii)$ take the view: either $y\in {\mathfrak p}$, where $\mathfrak p$ is a prime ideal of positive codimension, or there exists a Zariski-open subset $V\subset \widetilde S$ and a section $s' \in \sigma^{\ast}E (V)$ such that $V\supset U$, $s'|_U=s$ and  $s'|_{V\cap \widetilde S_0}\in$$\tors (\sigma^{\ast}E|_{\widetilde S^0})(V\cap \widetilde S^0)$. In the latter expression the torsion subsheaf  $\tors(\sigma^\ast E|_{\widetilde S^0})$ is understood in the usual sense.

The modified torsion subsheaf in each $ (\sigma_1\circ \dots \circ \sigma_i)^\ast E$ is defined in the analogous fashion. Also $\widetilde E_i=(\sigma_1\circ \dots \circ \sigma_i)^\ast E /\tors.$

For the sake of completeness we decsribe the subsheaf $\tors$ for the case of families with the arbitrary base $T$, although this description is similar to the construction for a single scheme.

Let $\Sigma$ be a family of schemes with fibres isomorphic to $S$ and with the base $T$, $\Sigma_i$ be a $T$-scheme with a natural decomposition  $\Sigma_i=\Sigma_i^0 \cup \Sigma_i^{add}$. The morphism  $\delta\!\!\delta_i$ factors through the exceptional divisor  $\D_i$ of the blowup morphism  $\sigma\!\!\!\sigma_i\colon \Sigma_i^0\to \Sigma_{i-1}$:
$$\delta\!\!\delta_i\colon \Sigma_i^{add} \to \D_i \hookrightarrow \Sigma_i^0.
$$
We know that $\delta\!\!\delta_i$ is flat over its scheme-theoretic image $\D_i$. Then one can write  (\ref{sh1}, \ref{sh2}) in the form $\widetilde \E=\sigma\!\!\!\sigma^\ast \E/\tors $, where the symbol   $\tors$ stands for the torsion in the modified sense. It is defined inductively as below. For convenience we preserve the notations $\widetilde \Sigma^0:=(\sigma\!\!\!\sigma_1^0\circ \dots \circ \sigma\!\!\!\sigma_\ell^{-\ell+1})^{-1}\Sigma$ for the principal component and   $\widetilde \Sigma^i=(\sigma\!\!\!\sigma_{i+1}^0\circ \dots \circ \sigma\!\!\!\sigma_\ell^{-\ell+2})^{-1}\Sigma_i^{add}$ for all $i=1, \dots, \ell-1$.

Let $U$ be a Zariski-open subset in one of the components $\widetilde \Sigma^0$ or $\widetilde \Sigma^j, j> 0$, and $\sigma\!\!\!\sigma^{\ast}\E|_{\widetilde \Sigma^j}(U)$ the corresponding group of sections. This group is a $\OO_{\widetilde \Sigma^j}(U)$-module. Let $\tors_j(U)$ be a submodule in  $\sigma\!\!\!\sigma^{\ast}\E|_{\widetilde \Sigma^j}(U)$ which consists of sections $s\in \sigma\!\!\!\sigma^{\ast}\E|_{\widetilde \Sigma^j}(U)$ such that  $s$ is annihilated by some prime ideal of positive codimension in ${\mathcal O}_{\widetilde \Sigma^j}(U)$. The correspondence $U \mapsto \tors_j(U)$ defines a subsheaf  $\tors_j \subset \sigma\!\!\!\sigma^{\ast}\E|_{\widetilde \Sigma^j}.$ Note that associated prime ideals of positive codimensions which annihilate sections $s\in \sigma\!\!\!\sigma^{\ast}\E|_{\widetilde \Sigma^j}(U)$, correspond to subschemes supported in  $\widetilde \Sigma^{add}=\bigcup_{j>0}\widetilde \Sigma^j.$ Since the scheme $\widetilde \Sigma=\widetilde \Sigma^0\cup \widetilde \Sigma^{add}$ is connected by the construction (because the additional components grown at any step intersect the principal component of this step), the subsheaves  $\tors_j, j\ge 0,$ allow to construct a subsheaf $\tors \subset \sigma^{\ast}\E$. The latter subsheaf is defined as follows. A section $s\in \sigma\!\!\!\sigma^{\ast}\E|_{\widetilde \Sigma^j}(U)$ satisfies the condition  $s\in \tors|_{\widetilde \Sigma^j}(U)$ if and only if \\
$(i)$ there exists a nonzero section $y\in \OO_{\widetilde \Sigma^j}(U)$ such that $ys=0$,\\
$(ii)$ for  $j>1$ at least one of requirements holds: either $y\in {\mathfrak p}$, where $\mathfrak p$ is a prime ideal of positive codimension, or there exist a Zariski-open subset  $V\subset \widetilde \Sigma$ and a section $s' \in \sigma\!\!\!\sigma^{\ast}\E (V)$ such that $V\supset U$, $s'|_U=s$, and for  $s'|_{V\cap \widetilde \Sigma^{j-1}}$ the requirements similar to  $(i), (ii)$ are fulfilled. For  $j=1$ the conditions $(ii)$ take the view: either $y\in {\mathfrak p}$, where $\mathfrak p$ is a prime ideal of positive codimension, or there is Zariski-open subset  $V\subset \widetilde \Sigma$ and a section $s' \in \sigma\!\!\!\sigma^{\ast}\E (V)$ such that  $V\supset U$, $s'|_U=s$ and $s'|_{V\cap \widetilde \Sigma^0}\in$$\tors (\sigma\!\!\!\sigma^{\ast}\E|_{\widetilde \Sigma^0})(V\cap \widetilde \Sigma^0)$. In the latter expression the subsheaf of torsion $\tors(\sigma\!\!\!\sigma^\ast \E|_{\widetilde \Sigma^0})$ is understood as follows:  for any open  $\widetilde U \subset \widetilde \Sigma^0$ $\tors(\sigma\!\!\!\sigma^\ast \E|_{\widetilde \Sigma^0})(\widetilde U)=\{s\in \sigma\!\!\!\sigma^\ast \E (\widetilde U)| \Supp s \subset \D \cap \widetilde U \}$. $\D$ stands for a closed subscheme in $\widetilde \Sigma^0$ where $\sigma\!\!\!\sigma|_{\widetilde \Sigma^0}$ is not an isomorphism.

Modified torsion subsheaves in each of sheaves $ (\sigma\!\!\!\sigma_1\circ \dots \circ \sigma\!\!\!\sigma_i)^\ast \E$ is defined in similar way. Also $\widetilde \E_i=(\sigma\!\!\!\sigma_1\circ \dots \circ \sigma\!\!\!\sigma_i)^\ast \E /\tors.$



\begin{thebibliography}{999}
\bibitem{Bar} 
V.~Baranovsky, {\it Uhlenbeck Compactification As a Functor} //
Int. Math. Res. Notices, V. 2015. No. 23. P. 12678--12712.

\bibitem{BruMarTikh}  U.~Bruzzo, D.~Markushevich, A.~Tikhomirov,
{\it  Uhlenbeck--Donaldson compactification for framed sheaves on
projective surfaces} // Math. Z. 2013. V. 275. No. 3-4. 
P. 1073--1093.

\bibitem{Donaldson89}  S.\,K.~Donaldson, {\it Compactification
and completion of Yang--Mills moduli spaces} in: {\it Differential
Geometry, Peniscola, 1988}, Springer,  Berlin, 1989.
P. 145--160 (Lecture Notes in Math, vol.1410) 

\bibitem{Eisen}   D.~Eisenbud, {\it Commutative algebra. With a
view toward algebraic geometry}, 
Springer-Verlag, New York--Berlin, 1995. Grad. Texts in Math., 150.


\bibitem{Feehan}  P.\,M.\,N.~Feehan, {\it Geometry of the ends
of the moduli space of anti-self-dual connections} // J. Diff.
Geom. 1995. V. 42. No. 3. P. 465--553.





\bibitem{Hart}  R.~Hartshorne, {\it Algebraic Geometry}. 
 New York--Heidelberg --
Berlin: Springer-Verlag, 1977.  Graduate Texts in Mathematics,  52.
\bibitem{HL}  D.~Huybrechts,  M.~Lehn, {\it The geometry of
moduli spaces of sheaves},  Vieweg, Braunschweig, 1997. Aspects
Math.,~E31.


\bibitem{Li93}  J.~Li, {\it Algebraic geometric interpretation
of Donaldson polynomial invariants} //  J. Diff. Geom., 1993, V. 37. No. 2. P. 417--466.

\bibitem{LuTel}  M.~L\"{u}bke, A.~Teleman, {\it The Kobayashi --
Hitchin correspondence}   World Scient. Publ. Co., 1995.

\bibitem{MTT}  D.~Markushevich, A.~Tikhomirov, G.~Trautmann,
{\it Bubble tree compactification of moduli spaces of vector
bundles on surfaces} // Cent. Eur. J. Math. 2012. V. 10. No. 4 
P. 1331--1355.

\bibitem{Mar1}  M.~Maruyama, {\it Moduli of stable sheaves, I}
// J. Math. Kyoto Univ. (JMKYAZ), 1977, V. 17. No. 1. P.
91--126.

\bibitem{Mar}  M.~Maruyama, {\it Moduli of stable sheaves, II}
// J. Math. Kyoto Univ. (JMKYAZ), 1978, V. 18. No. 3. P.
557--614.

\bibitem{RG}
 M.~Raynaud, L.~Gruson,
{\it Crit\`{e}res de
platitude et de projectivit\'{e}: Techniques de ``platification''
d'un module} // Inv. Math., 1971, V. 13.
P. 1--89.

\bibitem{Ray} 
 M.~Raynaud, 
{\it Flat modules in algebraic geometry},
// Compos. Math., 1972, 
V. 24.
no. 1. P. 11 -- 31



\bibitem{Tim0}
 N.\,V.~Timofeeva,  {\it Compactification in Hilbert scheme
of moduli scheme of stable 2-vector bundles on a surface} //  Math.
Notes,  {\bf 82:5}(2007), 677--690.

\bibitem{Tim1} 
N.\,V.~Timofeeva, {\it On a new compactification of the moduli of
vector bundles on a surface} // Sb. Math., {\bf 199:7}(2008),
1051--1070.
\bibitem{Tim2} 
N.\,V.~Timofeeva, {\it On a new compactification of the moduli of
vector bundles on a surface. II} // Sb. Math. {\bf 200:3}(2009),
405--427.

\bibitem{Tim3}
N.\,V.~Timofeeva, {\it On degeneration of surface in Fitting
compactification of moduli of stable vector bundles} // Math. Notes,
{\bf 90}(2011), 142--148.


\bibitem{Tim4}   N.\,V.~Timofeeva, {\it On a new
compactification of the moduli of vector bundles on a surface.
III: Functorial approach} //  Sb. Math., {\bf 202:3}(2011), 413 --
465.



\bibitem{Tim6}   N.\,V.~Timofeeva, {\it On a new
compactification of the moduli of vector bundles on a surface. IV:
Nonreduced moduli} // Sb. Math., {\bf 204:1}(2013), 133--153.




\bibitem{Tim8} N.\,V.~Timofeeva, {\it On a morphism of
compactifications of moduli scheme of vector bindles}  //  Siberian Electronic Mathematical Reports (SEMR) --
2015. -- V. 12. -- P.~577--591.

\bibitem{Tim9} N.\,V.~Timofeeva, {\it Isomorphism of
compactifications of moduli of vector bundles: nonreduced moduli}
arXiv:1411.7872 [math.AG]. 2014.  Russian original in: Modelirovanie i
Analiz Informatsionnykh Sistem (Modeling and Analysis of
Information Systems). 2015. V. 22. No. 5, P. 629 - 647.
\bibitem{Tim10} N.\,V.~Timofeeva,  {\it Admissible pairs vs Gieseker-Maruyama }
 // Sb. Math., 2019, V. 210. No. 5. P. 731--755.
\end{thebibliography}
\end{document}